\documentclass[11pt]{article}
\usepackage{amssymb,latexsym,amsmath,amsbsy,amsthm,amsxtra,amsgen,dsfont,pdfsync,amsfonts,graphicx}
\usepackage{xcolor}
\usepackage{hyperref}
\usepackage{enumitem}
\usepackage{tikz}

\newfont{\msbm}{msbm10 at 11pt}
\newcommand {\R} {\mbox{\msbm R}}
\def\eps{\varepsilon}
\def\MR#1{\href{http://www.ams.org/mathscinet-getitem?mr=#1}{MR#1}}

\newtheorem{Theo}{Theorem}
\newtheorem{Lemma}[Theo]{Lemma}
\newtheorem{Cor}[Theo]{Corollary}
\newtheorem{Prop}[Theo]{Proposition}
\newtheorem{Rmk}[Theo]{Remark}

\newcommand{\qedwhite}{\hfill \ensuremath{\Box}}

\newcounter{cnstcnt}
\newcommand{\newC}{%
\refstepcounter{cnstcnt}%
\ensuremath{C_{\thecnstcnt}}}
\newcommand{\oldC}[1]{\ensuremath{C_{\ref{#1}}}}

\newcounter{cnstcnta}
\newcommand{\newA}{%
\refstepcounter{cnstcnta}%
\ensuremath{A_\thecnstcnta}}
\newcommand{\oldA}[1]{\ensuremath{A_{\ref{#1}}}}

\newcounter{cnstcntx}
\newcommand{\newx}{%
\refstepcounter{cnstcntx}%
\ensuremath{\xi_\thecnstcntx}}
\newcommand{\oldx}[1]{\ensuremath{\xi_{\ref{#1}}}}

\newcommand{\footremember}[2]{%
    \footnote{#2}
    \newcounter{#1}
    \setcounter{#1}{\value{footnote}}%
}

\begin{document}
\title{Phase transition of the consistent maximal displacement of branching Brownian motion}
\author{Julien Berestycki \footremember{JB}{University of Oxford, E-mail: julien.berestycki@stats.ox.ac.uk}
\and Jiaqi Liu \footremember{JL}{University of Pennsylvania, E-mail: liujiaqi@sas.upenn.edu}
\and Bastien Mallein, \footremember{BM}{Institut de Mathématiques de Toulouse, UMR 5219, Université de Toulouse, UPS, F-31062 Toulouse Cedex 9, France. Email: {bastien.mallein@math.univ-toulouse.fr}}
\and Jason Schweinsberg, \footremember{JS}{University of California San Diego, Email: jschweinsberg@ucsd.edu}
}

\maketitle

\begin{abstract}
Consider branching Brownian motion in which we begin with one particle at the origin, particles independently move according to Brownian motion, and particles split into two at rate one.  It is well-known that the right-most particle at time $t$ will be near $\sqrt{2} t$.  Roberts considered the so-called consistent maximal displacement and showed that with high probability, there will be a particle at time $t$ whose ancestors stayed within a distance $ct^{1/3}$ of the curve $s \mapsto \sqrt{2} s$ for all $s \in [0, t]$, where $c = (3 \pi^2)^{1/3}/\sqrt{2}$.  We consider the question of how close the trajectory of a particle can stay to the curve $s \mapsto (\sqrt{2} + \eps) s$ for all $s \in [0, t]$, where $\eps > 0$ is small.  We find that there is a phase transition, with the behavior changing when $t$ is of the order $\eps^{-3/2}$.  This result allows us to determine, for branching Brownian motion in which particles have a drift to the left of $\sqrt{2} + \eps$ and are killed at the origin, the position at which a particle needs to begin at time zero for there to be a high probability that the process avoids extinction until time $t$.

\end{abstract}

{\small MSC: Primary 60J80; Secondary: 60J65, 60J25

Keywords: Branching Brownian motion, consistent maximal displacement}

\section{Introduction}
The ordinary branching Brownian motion (BBM) starts with one particle at the origin. This particle has an exponentially distributed lifetime with rate 1. During its lifetime, it moves according to standard Brownian motion. The lifetime is independent of its position. At the end of its lifetime, it undergoes dyadic branching. Each offspring independently repeats the above process and the system goes on.  Let $\mathcal{N}_t$ be the set of particles at time $t$ and $\{Y_u(t), u\in\mathcal{N}_t\}$ be the set of positions of particles at time $t$.  For $u \in \mathcal{N}_t$ and $s \in [0, t]$, we denote by $Y_u(s)$ the location of the particle at time $s$ that is the ancestor of the particle $u \in \mathcal{N}_t$.

There has been long-standing interest in the extremal position of the particles in BBM.
An asymptotic expression for the median position of the maximal displacement was first given by Bramson \cite{Bramson} in 1978.
Let $M_t = \max_{u \in \mathcal{N}_t} Y_u(t)$ be the position of the maximal particle at time $t$, and
denote by $m(t)$ the median position of $M_t$. Using a connection between the maximal displacement and the F-KPP equation, Bramson showed that as $t\to\infty$,
\begin{equation*}
m(t)=\sqrt{2}t-\frac{3}{2\sqrt{2}}\log t+O(1).
\end{equation*}
Bramson's proof was simplified by Roberts \cite{R13}, who also proved that almost surely
\[
  \liminf_{t\to\infty}\frac{M_t-\sqrt{2}t}{\log t} = -\frac{3}{2\sqrt{2}}, \qquad \limsup_{t\to\infty}\frac{M_t-\sqrt{2}t}{\log t}=-\frac{1}{2\sqrt{2}}.
\]
This second order almost sure limit result for the extremal position was first established by Hu and Shi \cite{HS09} in the discrete setting for branching random walks.

A natural follow-up question is how close can particles stay to the leading order of the maximal displacement $\sqrt{2}t$. This is called the consistent maximal displacement and has been studied in different contexts by Fang and Zeitouni \cite{Fang}, Faraud, Hu and Shi \cite{FHS12}, Jaffuel \cite{Jaffuel}, Mallein \cite{Mallein15, Mallein} and Roberts \cite{Roberts}. One way to characterize the consistent maximal displacement is to estimate the smallest deviation of particle trajectories from the critical line $s\mapsto \sqrt{2}s$ up to time $t$.  Equivalently, we can give the particles a drift to the left of $-\sqrt{2}$ and consider the smallest deviation of particle trajectories from the origin.  For $u \in \mathcal{N}_t$ and $\rho \in \R$, let $Y_u^{\rho}(t) = Y_u(t) - \rho t$, which gives the particles a drift of $-\rho$.  Define
\[
\mathcal{L}_{0}(t)=\inf_{u\in\mathcal{N}_t} \sup_{s\in [0, t]} - Y^{\sqrt{2}}_u(s).
\]
While the maximal displacement differs from $\sqrt{2}t$ on the scale of $\log t$, the consistent maximal displacement differs from $\sqrt{2} t$ on the scale of $t^{1/3}$. Let $c=(3\pi^2)^{1/3}/\sqrt{2}$. Fang and Zeitouni~\cite{Fang} and Faraud, Hu and Shi \cite{FHS12} proved in the discrete setting of branching random walks, and Roberts \cite{Roberts} stated in the continuous setting for ordinary BBM, that
\begin{equation}\label{consis0}
\lim_{t\to\infty} \frac{\mathcal{L}_0(t)}{c t^{1/3}}=1\quad \textup{a.s.}
\end{equation}
Furthermore, Theorem 2 in \cite{Roberts} implies the stronger result that $\{\mathcal{L}_{0}(t) - ct^{1/3}\}_{t>0}$ is tight.

For $\eps>0$, we can also define, see Figure \ref{figCmd} for illustration,
\[
\mathcal{L}_{\eps}(t) = \inf_{u\in\mathcal{N}_t} \sup_{s\in [0, t]} - Y_u^{\sqrt{2} + \eps}(s).
\]
It is easy to see that for $\eps>0$,
\begin{equation}\label{consisrho}
\lim_{t\to\infty} \frac{\mathcal{L}_\eps(t)}{\eps t}=1\qquad \textup{a.s.}
\end{equation}
Comparing equations (\ref{consis0}) and (\ref{consisrho}), we are interested in how the consistent maximal displacement transitions from $\eps t$ to $c t^{1/3}$ as $\eps$ approaches 0 from above.

\begin{figure}[ht]
\centering
\input{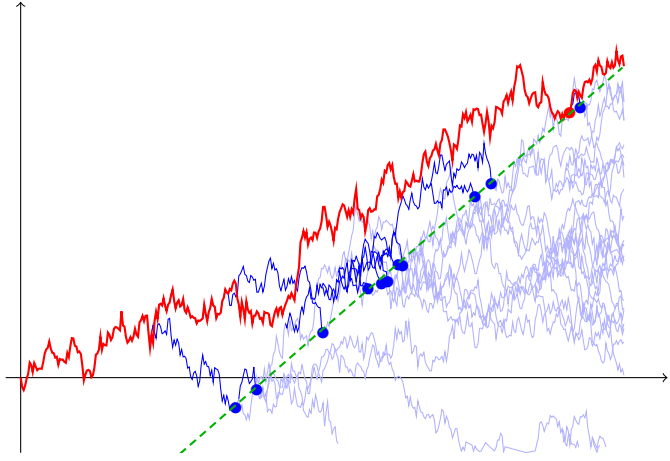}
\caption{Illustration of the definition of $\mathcal{L}_\eps(t)$ with $\eps = 0.1$. The path of a branching Brownian motion is drawn in blue, with the path of the particle realizing the minimum in the formula for $\mathcal{L}_\eps(t)$ in red. The line $s \mapsto (\sqrt{2} + \eps)s - \mathcal{L}_\eps(t)$ is drawn in green, so $-\mathcal{L}_\eps(t)$ corresponds to the intercept of this line with the origin. Note that all trajectories in the branching Brownian motion at time $t$ cross this line, with the red line being tangent.}
\label{figCmd}
\end{figure}

Let $$\omega=2^{-3/4}\pi.$$  Define the function $F: [0, \infty) \rightarrow [0, \infty)$ by
\[
F(u) = u-\omega\arctan\left(\frac{u}{\omega}\right).
\]
Note that $F'(u) = u^2/(u^2 + \omega^2) \in (0, 1)$ for all $u > 0$.  In particular, $F$ is strictly increasing, so the inverse $F^{-1}(u)$ is well defined for all $u\geq 0$.
The function $F$, up to scaling, appeared previously in the closely related work of Simon \cite{simon}; see equations (3.32) and (3.33) in \cite{simon}.  Define
\begin{equation}\label{L*def}
L^*_\eps(t)=\eps^{-1/2}F^{-1}(\eps^{3/2}t),
\end{equation}
and define
\begin{equation}\label{Lbardef}
\bar{L}_\eps(t) = L_\eps^*(t) + \frac{3}{2 \sqrt{2}} \log^+(\eps^{3/2} t),
\end{equation}
where $\log^+(x) = \max\{0, \log x\}$.
The following theorem shows that $\bar{L}_{\eps}(t)$ is a good approximation of $\mathcal{L}_{\eps}(t)$ until time $O(\eps^{-2})$.

\begin{Theo}\label{thmtight}
Let $\{t_\eps\}_{0 < \eps < 1}$ be a collection of times indexed by $\eps$.
Suppose there exists a positive constant $\newC\label{C_mainub}$ such that for all $\eps \in (0, 1)$, we have
\begin{equation}\label{tles}
0\leq t_\eps\leq \oldC{C_mainub} \eps^{-2}.
\end{equation}
Then the collection of random variables $\{\mathcal{L}_{\eps}(t_\eps)-\bar{L}_{\eps}(t_\eps)\}_{0<\eps<1}$ is tight.
\end{Theo}

Roughly speaking, for small $\eps$, we have $\bar{L}_\eps(t) \approx ct^{1/3}$ when $t \ll \eps^{-3/2}$ and $\bar{L}_\eps(t) \approx \eps t$ when $t\gg \eps^{-3/2}$.  We will give more precise asymptotics for $\bar{L}_{\eps}(t)$ in Section \ref{Lbarasym} below.  Theorem \ref{thmtight} thus implies that for $\eps$ sufficiently small, the consistent maximal displacement behaves like $\mathcal{L}_0(t)$, as in the critical case when $t \ll \eps^{-3/2}$, and grows linearly as in the supercritical case when $t \gg \eps^{-3/2}$, see Figure~\ref{figL}.  That is, we see a phase transition in the consistent maximal displacement on the time scale of $\eps^{-3/2}$.  We can also see from \eqref{Lbardef} that $\bar{L}_{\eps}(t_{\eps})$ could be replaced in the conclusion of Theorem \ref{thmtight} by the simpler expression $L_{\eps}^*(t_{\eps})$ if $\eps^{-2}$ were replaced by $\eps^{-3/2}$ in (\ref{tles}).

\begin{figure}[ht]
\centering
\begin{tikzpicture}[scale=1]
\draw[->] (0,0)--(8.5,0) node[right]{\tiny $t$};
\draw[->] (0,0)--(0,4)  node[left]{\tiny $L_\varepsilon(t)$};
\draw[scale=1, domain=0:2, smooth, thick, samples=200, variable=\x] plot ({\x}, {1.3*\x^(1/3)});
\draw[densely dotted] (2, 0) -- (2, 1.63);
\draw[scale=1, domain=5:8, smooth, variable=\x, thick] plot ({\x}, {\x/2});
\draw[densely dotted] (5, 0) -- (5, 2.5);
\draw[densely dotted] (8, 0) -- (8, 4);
\node at (1,1.6) {\tiny $\approx ct^{1/3}$};
\node at (1,-0.2) {\tiny $t \gg \varepsilon^{-3/2}$};
\node at (3.5,-0.2) {\tiny $t \asymp\varepsilon^{-3/2}$};
\node at (6.3,3.4) {\tiny $\approx \varepsilon t $};
\node at (6.8,-0.2) {\tiny $\varepsilon^{-3/2}\ll t \lesssim \varepsilon^{-2}$};

\coordinate (A) at (2, 1.63);
\coordinate (B) at (5, 2.5);
\draw[dotted] (A) to (B);
\node at (3.5,2.6) {\tiny $\approx \varepsilon^{-1/2}F^{-1}(\varepsilon^{3/2}t)$};
\end{tikzpicture}
\caption{Illustration of the graph of $t \mapsto L_\varepsilon(t)$ at the different timescales.}
\label{figL}
\end{figure}
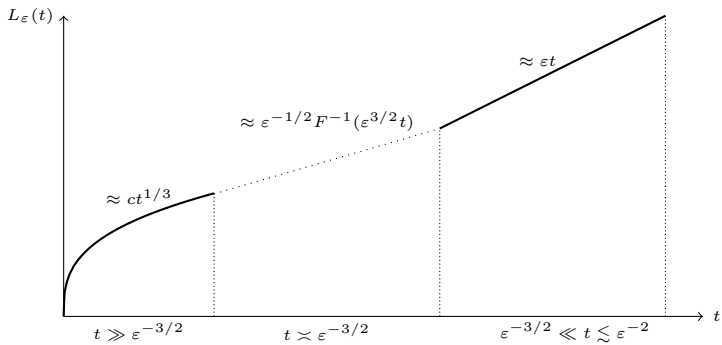

The consistent maximal displacement problem can also be formulated in terms of BBM with an absorbing barrier. Indeed, this is the setup considered by Jaffuel \cite{Jaffuel}. Consider BBM with absorption in which each particle moves as Brownian motion with drift $-\rho$, splits into two particles at rate $1$ and is killed upon hitting the origin. Kesten \cite{kesten} first studied this model in 1978 and delineated the regions where the process is subcritical, critical or supercritical.  Kesten proved that $\rho=\sqrt{2}$ is the critical value separating the supercritical case $\rho<\sqrt{2}$ and the subcritical case $\rho>\sqrt{2}$. We focus on the slightly subcritical case where $\rho$ is slightly above the critical value $\sqrt{2}$.
Write $\rho=\sqrt{2}+\eps$ for $0<\eps<1$. We use $P_x$ to represent the probability measure of BBM with absorption which starts from a single particle at $x$ at time $0$ and has drift $-\rho$. Let $\zeta$ be the extinction time.

To see the connection with consistent maximal displacement, note that $\mathcal{L}_{\eps}(t) < x$ means there is a particle $u \in \mathcal{N}_t$ such that $- Y_u^{\sqrt{2} + \eps}(s) < x$ for all $s \in [0,t]$, which means that $x + Y^{\sqrt{2} + \eps}_u(s) > 0$ for all $s \in [0, t]$.  Therefore, if we consider a process starting from an initial particle at $x$ such that all particles have a drift of $-(\sqrt{2} + \eps)$, the survival of the process until time $t$ is equivalent to the condition that $\mathcal{L}_{\eps}(t) < x$.  Let $\zeta$ be the extinction time. We have
\begin{equation}\label{Lepszeta}
P(\mathcal{L}_{\eps}(t) < x) = P_x(\zeta > t).
\end{equation}
The following theorem states that $\bar{L}_{\eps}(t)$ is also the position where a particle needs to be at time zero in order for it to have a good chance of surviving until time $t$.

\begin{Theo} \label{thmLzt}
Suppose (\ref{tles}) holds. For every ~$\delta>0$, there exist positive constants $\newC\label{C_main1}$ and $\newC\label{C_main2}$ such that for all $\eps \in (0, 1)$,
\begin{equation}\label{Lztlb}
P_{\bar{L}_\eps(t_\eps)-\oldC{C_main1}}(\zeta>t_\eps)<\delta
\end{equation}
and
\begin{equation}\label{newLztub}
P_{\bar{L}_\eps(t_\eps)+\oldC{C_main2}}(\zeta<t_\eps)<\delta.
\end{equation}
\end{Theo}

Theorem \ref{thmLzt} is a reformulation of Theorem \ref{thmtight} in the setting of BBM with absorption.  Indeed, equation \eqref{Lepszeta} implies that \eqref{Lztlb} is equivalent to the statement that $P(\mathcal{L}_{\eps}(t_{\eps}) < \bar{L}_{\eps}(t_{\eps}) -\oldC{C_main1}) < \delta$, and \eqref{newLztub} is equivalent to the statement that $P(\mathcal{L}_{\eps}(t_{\eps}) \geq \bar{L}_{\eps}(t_{\eps}) + \oldC{C_main2}) < \delta$.  Thus, Theorem \ref{thmLzt} immediately implies the tightness in Theorem \ref{thmtight}, and we will focus for the rest of the paper on proving Theorem~\ref{thmLzt}.

In the critical case, the relevance of the curve $ct^{1/3}$ was already apparent from the work of Kesten \cite{kesten}.  Berestcyki, Berestycki and Schweinsberg \cite{bbs14} proved that in the case of critical drift, $ct^{1/3}$ is roughly the position where a particle needs to be located at time zero in order for it to have a good chance of surviving until time~$t$. Maillard and Schweinsberg later obtained a stronger result; see Theorem 1.3 of \cite{ms20}. In the subcritical case $\rho>\sqrt{2}$, it is obvious that particles need to start from at least $\eps t$ in order to have descendants alive until time $t$. In the slightly subcritical case, Theorem~\ref{thmLzt} shows that to have a good chance of survival until time $t$, the position of the initial particle should be around $ct^{1/3}$ if $t\ll\eps^{-3/2}$ and around $\eps t$ when $t \gg \eps^{-3/2}$. Therefore, $\eps^{-3/2}$ is the time scale when the slightly subcritical process transitions from critical behavior to subcritical behavior. For $t \gg \eps^{-2}$, the problem of obtaining tight estimates for the position where a particle must start to have a descendant alive at time $t$ remains open.

Observe that one can also frame these results in terms of PDEs. More precisely, if one writes  $u(x,t)  = P_x(\zeta\le t)$ then $u$ satisfies the following boundary value problem:
\[
\left\{
\begin{array}{l}
\partial_t u =\frac12 \partial^2_{xx} u  -\rho \partial_x u +u(1-u) \\
u(0,t)=1 \, \forall t\ge 0 \text { and } u(0,x) = 0 \, \forall x >0.
\end{array}
\right.
\]
Then Theorem \ref{thmLzt} above suggests that $u(x,t)$ develops a front whose position will initially be at $ct^{1/3}$ before times of order $\eps^{-3/2}$ while after times $t\gg \eps^{-2}$ the position will be near $\eps t$. The shape of the front can be expected to be that of the usual critical travelling wave of the Fisher-KPP equation. In the critial case $\rho =\sqrt 2$ the convergence $u(ct^{1/3}+x,t )$ to the travelling wave was proven in Theorem 1.3 of \cite{ms20}.

\subsection{Main ingredients of the proof}

In this subsection, we give an overview of the strategy of the proof.  We will first introduce some notation that will be used throughout the paper. For two collections of positive numbers $\{x_{\eps}\}_{0<\eps<1}$ and $\{y_\eps\}_{0<\eps<1}$ indexed by $\eps$, if $x_{\eps}/y_{\eps}$ is bounded above by a positive constant, we write $x_{\eps}\lesssim y_{\eps}$. If $\lim_{\eps\downarrow0}x_{\eps}/y_{\eps}=0$, we write $x_{\eps}\ll y_{\eps}$. We define $x_{\eps}\gtrsim y_{\eps}$ and $x_{\eps}\gg y_{\eps}$ accordingly. Moreover, the notation $x_{\eps}\asymp y_{\eps}$ means that $x_{\eps}/y_{\eps}$ is bounded both above and below by positive constants.

We now define a function which is a perturbation of the function $F$ defined above.
We define $F_{\eps}: [0, \infty) \rightarrow \R$ by
\[
F_\eps(u) = u-\omega\arctan\left(\frac{u}{\omega}\right)-\frac{3}{4\sqrt{2}}\eps^{1/2}\log\left(\frac{u^2}{\omega^2}+1\right).
\]
Note that $F_{\eps}(0) = 0$ and $\lim_{u \rightarrow \infty} F_{\eps}(u) = \infty$.  Also, $$F_{\eps}'(u) = \frac{1}{\omega^2 + u^2} \bigg(u^2 - \frac{3}{2 \sqrt{2}} \eps^{1/2} u \bigg).$$
Because $F'_{\eps}(u) < 0$ for $u < \frac{3}{2 \sqrt{2}} \eps^{1/2}$ and $F'_{\eps}(u) > 0$ for $u > \frac{3}{2 \sqrt{2}} \eps^{1/2}$, we see that for each $t > 0$, there is a unique $u > 0$ such that $F_{\eps}(u) = t$.  Therefore, the inverse function $F_{\eps}^{-1}(t)$ is well-defined for $t > 0$, and we can define $F_{\eps}^{-1}(0) = \lim_{t \downarrow 0} F_{\eps}^{-1}(t)$.
For $0 < \eps < 1$ and $t > 0$, we define
\begin{equation}\label{defL}
L_{\eps}(t) = \eps^{-1/2}F_{\eps}^{-1}(\eps^{3/2}t).
\end{equation}
Lemma \ref{LepsLbar} below shows that $L_{\eps}(t)$ and $\bar{L}_{\eps}(t)$ differ only by a constant, so it is enough to prove Theorem \ref{thmLzt} with $L_{\eps}(t_{\eps})$ in place of $\bar{L}_{\eps}(t_{\eps})$.

We will need to consider a modified process in which particles are killed not only when they reach the origin but also when they reach an upper boundary.
For $0\leq s < t_{\eps}$, define
\[
H_\eps(s)=L_\eps(t_{\eps}-s),
\]
\[
K_\eps(s)=L_\eps^*(t_{\eps}-s).
\]
To prove equation (\ref{Lztlb}), we consider a process in which we kill particles both at the origin and at $H_\eps(\cdot)$.  For the original process to survive until time $t_{\eps}$, either the modified process must have particles surviving until $t_{\eps}$, or at least one particle must hit $H_\eps(\cdot)$ in the modified process.  By Markov's inequality, the probability that at least one particle hits $H_{\eps}(\cdot)$ is bounded above by the expected number of particles hitting $H_{\eps}(\cdot)$.  Therefore, the key to proving (\ref{Lztlb}) is to estimate the first moment of the number of particles hitting $H_\eps(\cdot)$ in the modified process.

To prove equation (\ref{newLztub}), we consider the cases $t_\eps\lesssim\eps^{-3/2}$ and $\eps^{-3/2}\ll t_\eps\lesssim\eps^{-2}$ separately. When $t_\eps\lesssim\eps^{-3/2}$, the function $L_\eps(t_\eps)$ is within $O(1)$ distance of $L_\eps^*(t_\eps)$. We first show that equation~(\ref{newLztub}) holds with $L_\eps(t_\eps)$ replaced by $L_\eps^*(t_\eps)$, which is Proposition \ref{propLzt} below.  For this proposition and the following one, we will fix a positive constant $\newC\label{C_shorttime}>2.$

\begin{Prop}\label{propLzt}
Let $\{t_\eps\}_{0 < \eps < 1}$ be a collection of times indexed by $\eps$. Suppose for all $\eps \in (0, 1)$,
\begin{equation}\label{limtz}
0 \leq t_{\eps} \leq \oldC{C_shorttime} \eps^{-3/2}.
\end{equation}
For every ~$\delta>0$, there exists a positive constant $\newC\label{C_th2}$ that only depends on $\oldC{C_shorttime}$ and $\delta$ such that for all $\eps$ sufficiently small,
 \begin{equation}\label{Lztub}
P_{L_\eps^*(t_\eps)+\oldC{C_th2}}(\zeta<t_\eps)<\delta.
\end{equation}
\end{Prop}

To prove Proposition \ref{propLzt}, we stop particles when they reach an upper boundary $K_{\eps}(\cdot)$ and use an argument similar to one in \cite{bbs14}.  We use first and second moment estimates of the number of particles hitting $K_\eps(\cdot)$ to show that if a particle starts near $L^*_{\eps}(t_{\eps})$, then descendants of this particle are likely to reach the upper boundary at a later time.  Then descendants of those particles are likely to hit the upper boundary again at a still later time, and so on.  This leads to a natural coupling with a supercritical branching process and establishes that the process is likely to survive until time $t_{\eps}$.

When $\eps^{-3/2}\ll t_\eps\lesssim \eps^{-2}$, we use the curve $H_{\eps}(\cdot)$ as the upper boundary.  We show that if the process starts with one particle near $L_\eps(t_\eps)$, then the probability that at least one descendant of this particle hits the curve $H_\eps(\cdot)$ during a time interval that is $O(\eps^{-3/2})$ before time $t_{\eps}$ is bounded from below by a positive constant.  This result is stated precisely in Proposition \ref{prop-2} below.

\begin{Prop}\label{prop-2}
Let $\{t_\eps\}_{0 < \eps < 1}$ be a collection of times indexed by $\eps$.
Suppose for all $\eps \in (0, 1)$,
\begin{equation}\label{limtzinf}
\oldC{C_shorttime} \eps^{-3/2} \leq t_{\eps} \leq \oldC{C_mainub} \eps^{-2}.
\end{equation}
Fix $\oldC{C_shorttime}/2<\mu<\lambda<\oldC{C_shorttime}-1$.  Consider the process which starts from a single particle at $L_\eps(t_\eps)-1$. There exists a positive constant $\newC\label{C_longhit}$ that only depends on $\oldC{C_mainub}$ such that for $\eps$ sufficiently small, the probability that at least one particle hits $H_\eps(\cdot)$ during the time interval $(t_\eps-\lambda\eps^{-3/2}, t_\eps-\mu\eps^{-3/2})$ is bounded below by $\oldC{C_longhit}$.
\end{Prop}

Proposition~\ref{prop-2} then implies that if the process starts from one particle well above $L_\eps(t_{\eps})$, the probability that there is at least one descendant hitting the curve $H_\eps(\cdot)$ during this time interval is close to 1. Equation (\ref{newLztub}) follows from this fact combined with Proposition \ref{propLzt}. The proof of Proposition \ref{prop-2} is based on first and second moment estimates of the number of particles hitting $H_\eps(\cdot)$.

\subsection{Heuristic derivation of $L_\eps^*(\cdot)$}

We explain in this subsection how the curve $L^*_{\eps}(\cdot)$ can be understood.  We first consider a modified BBM process in which particles move as Brownian motion with drift $-\rho$, and particles are killed not only at the origin but also at some level $K$, which for now we take to be a fixed constant.  We call this process BBM in the strip.  Denote by $q_s^K(x,y)$ the density for this process, by which we mean that if there is one particle at $x$ at time zero, then the expected number of descendants of this particle in the set $A$ at time $t$ is given by $\int_A q_s^K(x,y) \: dy$.  According to equation (12) in \cite{bbs13}, if $s$ is large enough, then the density can be approximated by
\begin{equation}\label{ptxy}
q_s^{K}(x,y) \approx \frac{2}{K} e^{(1 - \rho^2/2 - \pi^2/2K^2)s} e^{\rho x} \sin \bigg(\frac{\pi x}{K} \bigg) e^{-\rho y} \sin \bigg(\frac{\pi y}{K} \bigg).
\end{equation}
This formula indicates that after a sufficiently long time, the number of particles near $y$ is proportional to $e^{-\rho y} \sin(\pi y/K)$.

The density of BBM in the strip will give a good approximation to the particle configuration of BBM with absorption if $K$ is large enough that relatively few particles reach the upper boundary but small enough that descendants of rare particles that drift close to the upper boundary do not dominate the first moment calculations.  These considerations require that we set $K$ to be close to where we expect the particle furthest from the origin to be located.  Note that if we place $N$ particles independently according to the exponential distribution with density $\rho e^{-\rho y}$, then the particle farthest from the origin will be located near $(\log N)/\rho$.  However, because the exponential term in (\ref{ptxy}) suggests that the number of particles decreases over time, the location of the upper boundary needs to move closer to the origin over time. Writing $\rho=\sqrt{2}+\eps$, we therefore replace the fixed constant $K$ in (\ref{ptxy}) by a function $K_{\eps}(\cdot)$.  We also let $N_{\eps}(s)$ denote the number of particles at time $s$.  Differentiating (\ref{ptxy}) with respect to $s$ and integrating over the density to give an estimate of the total number of particles then yields the rough approximation
$$N_{\eps}'(s) \approx \left(1-\frac{\rho^2}{2}- \frac{\pi^2}{2K_\eps^2(s)}\right) N_{\eps}(s).$$
If we choose the curve $K_{\eps}(s)$ so that $K_{\eps}(s) \approx (\log N_{\eps}(s))/\rho$, then
\begin{equation}\label{Keq}
K_{\eps}'(s) \approx \frac{1}{\rho N_{\eps}(s)} N_{\eps}'(s) \approx \frac{1}{\rho} \bigg(1 - \frac{\rho^2}{2} - \frac{\pi^2}{2K_{\eps}(s)^2} \bigg).
\end{equation}

Consider first the critical case in which $\rho = \sqrt{2}$.  One can easily check that the solution to this differential equation when $K_0(t) = 0$ is $$K_0(s) = c(t - s)^{1/3}.$$
which is the curve used for truncation in \cite{bbs14, kesten, ms20}.  This calculation is consistent with the result in the critical case that $K_0(0) = ct^{1/3}$ is where a particle must begin to have a good chance to have a descendant alive at time $t$.  When $\eps > 0$, if we discard some small terms, the differential equation (\ref{Keq}) becomes
$$K_{\eps}'(s) \approx -\eps - \frac{\pi^2}{2 \sqrt{2}} \cdot \frac{1}{K_{\eps}(s)^2}.$$  Dividing both sides by $\eps + \pi^2/(2 \sqrt{2} K_{\eps}(s)^2)$, then integrating from $0$ to $t$ and making the change of variables $y = K_{\eps}(s)$, we get
\begin{equation}\label{Fderive}
-t = \int_0^t \frac{K_{\eps}'(s)}{\eps + \frac{\pi^2}{2 \sqrt{2} K_{\eps}(s)^2}} \: ds = \frac{1}{\eps} \int_{K_{\eps}(0)}^{K_{\eps}(t)} \frac{y^2}{y^2 + \frac{\pi^2}{2 \sqrt{2} \eps}} \: dy.
\end{equation}
Recalling that $\omega = 2^{-3/4} \pi$ and integrating, we obtain
$$-\eps t = K_{\eps}(t) - K_{\eps}(0) - \omega \eps^{-1/2} \bigg[ \arctan \bigg( \frac{\eps^{1/2} K_{\eps}(t)}{\omega} \bigg) - \arctan \bigg( \frac{\eps^{1/2} K_{\eps}(0)}{\omega} \bigg) \bigg].$$
It follows that if $K_{\eps}(t) = 0$, then
\begin{equation}\label{K0eq}
K_{\eps}(0) = \eps t + \omega \eps^{-1/2} \arctan \bigg( \frac{\eps^{1/2} K_{\eps}(0)}{\omega} \bigg).
\end{equation}
Now, if we let $F(u) = u - \omega \arctan(u/\omega)$, then (\ref{K0eq}) implies that $F(\eps^{1/2} K_{\eps}(0)) = \eps^{3/2} t$, and therefore $K_{\eps}(0) = \eps^{-1/2} F^{-1}(\eps^{3/2} t)$, matching our definition of $L^*_{\eps}(t)$ in (\ref{L*def}).

The above heuristics are accurate when $t \lesssim \eps^{-3/2}$.  For $t \gg \eps^{-3/2}$, a finer adjustment is needed.  We will replace $K_{\eps}(\cdot)$ by $H_{\eps}(\cdot)$ so that the curve will satisfy the result of Lemma \ref{newtau(r,s)} rather than the result of Lemma \ref{tau(r,s)} below.

\begin{Rmk}
{\em One could also consider the asymptotics of $\mathcal{L}_{\eps}(t)$ when $\eps < 0$.  In this case, it is straightforward to see that $\mathcal{L}_{\eps}(t)$ converges to a finite limit as $t \rightarrow \infty$.  The heuristics leading to (\ref{Fderive}) still hold when $\eps < 0$.  Therefore, when $\eps < 0$, we derive
\begin{align*}
|\eps| t &= - \int_{K_{\eps}(0)}^{K_{\eps}(t)} \frac{y^2}{\frac{\pi^2}{2 \sqrt{2} |\eps|} - y^2} \: dy \\
&= \omega |\eps|^{-1/2} \tanh^{-1} \bigg( \frac{K_{\eps}(0)}{\omega |\eps|^{-1/2}} \bigg) - \omega |\eps|^{-1/2} \tanh^{-1} \bigg( \frac{K_{\eps}(t)}{\omega |\eps|^{-1/2}} \bigg) - K_{\eps}(0) + K_{\eps}(t).
\end{align*}
It follows that if $K_{\eps}(t) = 0$, then
$$|\eps|^{3/2} t = \omega \tanh^{-1} \bigg( \frac{|\eps|^{1/2} K_{\eps}(0)}{\omega} \bigg) - |\eps|^{1/2} K_{\eps}(0).$$
If we define the function $G(u) = \omega \tanh^{-1}(u/\omega) - u$, then $|\eps|^{3/2} t = G(|\eps|^{1/2} K_{\eps}(0))$, and therefore $K_{\eps}(0) = |\eps|^{-1/2} G^{-1}(|\eps|^{3/2} t).$  We therefore conjecture that if we define $\bar{L}_{\eps}(t) = |\eps|^{-1/2} G^{-1}(|\eps|^{3/2} t)$ when $\eps < 0$, then a result similar to Theorem \ref{thmtight} should hold for negative $\eps$.  Note also that $\lim_{u \rightarrow \infty} G^{-1}(u) = \omega$, which means that for $t \gg |\eps|^{-3/2}$, we have $\bar{L}_{\eps}(t) \approx |\eps|^{-1/2} \omega$.  Indeed, it was shown in \cite{bbs11} that for BBM with absorption with drift $\sqrt{2 - \delta}$ for small $\delta$, the initial particle needs to start near $\pi/\sqrt{\delta}$ for the process to have a good chance of surviving forever.  Because $\delta \approx -2 \sqrt{2} \eps$, this corresponds to $|\eps|^{-1/2} \omega$ in our parameterization.  See also \cite{ghs} for results on branching random walks with a slightly supercritical drift.  However, we do not pursue the case of $\eps < 0$ further in this paper.}
\end{Rmk}

\subsection{Organization of the paper}

The rest of this paper is organized as follows. In Section \ref{densitywithboundary}, we consider two modified BBM processes in which particles are killed not only at the origin, but also at a right boundary curve. In the first process,  particles are killed at either the origin or $H_\eps(\cdot)$. In the second process, particles are killed at either the origin or $K_\eps(\cdot)$. We approximate the densities of the two modified processes. In Section \ref{numberhitting}, we estimate the first and second moments of the number of particles hitting the right boundaries in the two modified processes. Finally, in Section \ref{pfpropLsurviv}, after proving Propositions \ref{propLzt} and \ref{prop-2}, we give the proof of Theorem  \ref{thmLzt}.

To lighten the burden of notation, we will usually omit the subscript $\eps$ in the notation that is related to time. For example, we will write $t$ in place of $t_{\eps}$. However, it is important to keep in mind that we are considering a collection of processes indexed by $\eps$ and the time $t$ does depend on $\eps$.

\section{Density with killing at the right boundary}\label{densitywithboundary}

\subsection{Asymptotics for $\bar{L}_{\eps}(t)$}\label{Lbarasym}

In this section, we record some asymptotic expressions for $\bar{L}_{\eps}(t)$ which are useful for interpreting the main results.

\bigskip
\noindent {\bf The case $t \ll \eps^{-3/2}$}:  Using the Taylor expansion of $\arctan(x)$ at $x = 0$, we get
$$F(u) = u - \omega \arctan \bigg( \frac{u}{\omega} \bigg) = \frac{u^3}{3 \omega^2} - \frac{u^5}{5 \omega^4} + O(u^7) \quad\mbox{as }u \rightarrow 0.$$
Noting that $c = 3^{1/3} \omega^{2/3}$, we get after some algebra,
\begin{equation}\label{F-1asymp3/5}
F^{-1}(u) = cu^{1/3} + \frac{3u}{5} + O(u^{5/3})\quad\mbox{as }u \rightarrow 0.
\end{equation}
When $t \ll \eps^{-3/2}$, equations \eqref{L*def} and \eqref{Lbardef} give $\bar{L}_{\eps}(t) = \eps^{-1/2} F^{-1}(\eps^{3/2} t) + O(1)$, which means
$$\bar{L}_{\eps}(t) = ct^{1/3} + \frac{3}{5} \eps t + O(\eps^2 t^{5/3}) + O(1).$$
The first term dominates when $t \ll \eps^{-3/2}$.  When $t \lesssim \eps^{-1}$, the right-hand side simplifies to $ct^{1/3} + O(1)$, exactly as in the critical case when $\eps = 0$.

\bigskip
\noindent {\bf The case $t = \lambda \eps^{-3/2}$}:  In this case, we have $$\bar{L}_{\eps}(t) = \eps^{-1/2} F^{-1}(\lambda) + O(1).$$
One can write this as $\bar{L}_{\eps}(t) = \lambda^{-1/3} F^{-1}(\lambda) t^{1/3} + O(1).$
One can then see that $\lambda^{-1/3} F^{-1}(\lambda) \rightarrow c$ as $\lambda \rightarrow 0$, and $\lambda^{-1/3} F^{-1}(\lambda) \sim \lambda^{2/3}$ as $\lambda \rightarrow \infty$.

\bigskip
\noindent {\bf The case $\eps^{-3/2} \ll t \lesssim \eps^{-2}$}:  When $u$ is large, we can use the identity $\arctan(x) + \arctan(1/x) = \pi/2$ for $x>0$ to get $$F(u) = u - \omega \arctan \bigg( \frac{u}{\omega} \bigg) = u - \frac{\omega \pi}{2} + \frac{\omega^2}{u} + O(u^{-3}) \quad\mbox{as }u \rightarrow \infty.$$  It follows that
$$F^{-1}(u) = u + \frac{\omega \pi}{2} + O(u^{-1}).$$
We therefore have $$\bar{L}_{\eps}(t) = \eps t + \frac{\omega \pi}{2} \eps^{-1/2} + O(\eps^{-2} t^{-1}) + \frac{3}{2 \sqrt{2}} \log^+(\eps^{3/2} t).$$
Note that the third term on the right-hand side is smaller than the second term but may be larger than the fourth term.

\bigskip
\noindent {\bf The case $t \gg \eps^{-2}$}:  According to Theorem \ref{thmtight},
we do not know that $\mathcal{L}_{\eps}(t)$ is close to $L_{\eps}(t)$ when $t \gg \eps^{-2}$.  It remains an open question to find the correct value of $\mathcal{L}_{\eps}(t)$ in this case.  We can easily obtain upper and lower bounds.  Consider BBM with drift $-(\sqrt{2} + \eps)$.  If there were no killing at the origin, results of Bramson \cite{Bramson} establish that the right-most descendants of a particle at $x$ would be located near $x - \eps t - \frac{3}{2 \sqrt{2}} \log t$.  Therefore, we must have $${\cal L}_{\eps}(t) \geq \eps t + \frac{3}{2 \sqrt{2}} \log t + O_p(1)$$ because even with no killing at the origin, a descendant particle would not end up above the origin at time $t$ if it started below the expression on the right-hand side.  On the other hand, suppose instead of killing particles at the origin, we kill particles at time $s$ if they reach $\eps(t - s)$.  Then we have a translation of the problem in which the drift is critical and particles are killed at the origin.  The results on the critical case show that a particle must start out at $\eps t + ct^{1/3} + O(1)$ to have a good chance of having a descendant alive at time $t$.  It follows that $$\mathcal{L}_{\eps}(t) \leq \eps t + ct^{1/3} + O_p(1).$$

\subsection{Properties of $L_\eps(t)$, $L_\eps^*(t)$ and related functions}

Note that $F_{\eps}(u) < F(u) < u$ for all $u > 0$ and all $\eps \in (0, 1)$.  Therefore,
\begin{equation}\label{Fcompare}
u < F^{-1}(u) < F_{\eps}^{-1}(u)
\end{equation}
for all $u > 0$ and all $\eps \in (0, 1)$.  Furthermore, we have
\begin{equation}\label{F-1epsasymp}
\lim_{u \rightarrow \infty} u^{-1} F(u) = \lim_{u \rightarrow \infty}u^{-1} F_{\eps}(u) = 1
\end{equation}
for all $\eps \in (0, 1)$, and the convergence is uniform over $\eps \in (0, 1)$.  Therefore, there exists $u^* > 0$ such that $F_{\eps}(u) > u/2$ for all $u \geq u^*$ and all $\eps \in (0, 1)$, and therefore $u < F_{\eps}^{-1}(u) < 2u$ for all $u \geq u^*$.  This implies that when $t \geq u^* \eps^{-3/2}$,
\begin{equation}\label{L_eps<2epst}
\eps t \leq L_\eps(t)\leq 2\eps t.
\end{equation}
Note also that if we write $a = \frac{3}{2 \sqrt{2}} \eps^{-1/2}$, then by applying the Fundamental Theorem of Calculus and splitting up the integral based on whether the integrand is negative or positive, we get
\begin{align*}
F_{\eps}\left(\frac{3a}{2} \right) &= \int_0^{3a/2} F_{\eps}'(u) \: du \\
&= \int_0^{a} \frac{1}{\omega^2 + u^2}(u^2 - au) \: du + \int_a^{3a/2} \frac{1}{\omega^2 + u^2}(u^2 - au) \: du \\
&\leq \frac{1}{\omega^2 + a^2} \int_0^a (u^2 - au) \: du + \frac{1}{\omega^2 + a^2} \int_a^{3a/2} (u^2 - au) \: du \\
&= 0.
\end{align*}
Therefore, $F_{\eps}^{-1}(0) \geq \frac{3a}{2} = \frac{9}{4 \sqrt{2}} \eps^{1/2}$, and so
\begin{equation}\label{Leps0}
L_{\eps}(0) \geq \frac{9}{4 \sqrt{2}}.
\end{equation}

The next result establishes that $L_{\eps}(t)$ and $\bar{L}_{\eps}(t)$ differ only by a constant, which allows us to use $L_{\eps}(t)$ in place of $\bar{L}_{\eps}(t)$ in the remainder of the proof.

\begin{Lemma}\label{LepsLbar}
We have $$\eps^{-1/2} (F_{\eps}^{-1}(u) - F^{-1}(u)) = \frac{3}{2 \sqrt{2}} \log^+(u) + O(1),$$
where the $O(1)$ term is uniformly bounded over all $u > 0$ and $\eps \in (0,1)$.
\end{Lemma}

\noindent\textit{Proof.}
Because $F(F^{-1}(u)) = F_{\eps}(F_{\eps}^{-1}(u)) = u$, we have
\begin{equation}\label{FFeps}
F(F^{-1}(u)) = F(F_{\eps}^{-1}(u)) - \frac{3}{4 \sqrt{2}} \eps^{1/2} \log \bigg( \frac{F_{\eps}^{-1}(u)^2}{\omega^2} + 1 \bigg).
\end{equation}
We now consider two cases.  First, suppose $u \leq u^*$.  Write $Z = F^{-1}_1(u^*)^2 + \omega^2$.
Observe that for all $x \leq F^{-1}_\eps(u^*)$ and $\eps \in (0,1)$, because $\eps \mapsto F^{-1}_\eps(u^*)$ is nondecreasing, we have
\[
  F'(x)  = \frac{x^2}{x^2+ \omega^2} \geq \frac{x^2}{F^{-1}_\eps(u^*)^2 + \omega^2} \geq \frac{x^2}{Z}.
\]
Then, we have
\[
  F(F^{-1}_\eps(u)) - F(F^{-1}(u)) = \int_{F^{-1}(u)}^{F^{-1}_\eps(u)} F'(x) dx \geq \frac{1}{3Z} \left(F^{-1}_\eps(u)^3 - F^{-1}(u)^3\right).
\]
Note that for $0 < a < b$, we have $b^3 - a^3 = (b-a) (a^2 + ab + b^2) \geq (b-a) b^2$.  Therefore \eqref{FFeps} yields
\[
  \frac{3}{4 \sqrt{2}} \eps^{1/2} \log \bigg( \frac{F_{\eps}^{-1}(u)^2}{\omega^2} + 1 \bigg) \geq \frac{F^{-1}_\eps(u)^2}{3Z} (F_\eps^{-1}(u) - F(u)).
\]
Rearranging this inequality, and then using that $\log(1 + x) \leq x$ for $x \geq 0$, we have
$$\eps^{-1/2} (F_{\eps}^{-1}(u) - F^{-1}(u)) \leq \frac{9 Z}{4 \sqrt{2}} \cdot \frac{1}{F_{\eps}^{-1}(u)^2} \log \bigg( \frac{F_{\eps}^{-1}(u)^2}{\omega^2} + 1 \bigg) \leq \frac{9Z}{4 \sqrt{2}\omega^2},$$
which proves the result in the case when $u \leq u^*$.

Now, suppose $u > u^*$.  By the Mean Value Theorem, there exists $v \in [F^{-1}(u), F_{\eps}^{-1}(u)]$ such that $$F(F_{\eps}^{-1}(u)) - F(F^{-1}(u)) = F'(v)(F_{\eps}^{-1}(u) - F^{-1}(u)).$$  Combining this result with (\ref{FFeps}), we get
$$\eps^{-1/2}( F_{\eps}^{-1}(u) - F^{-1}(u)) = \frac{3}{4 \sqrt{2} F'(v)} \log \bigg( \frac{F_{\eps}^{-1}(u)^2}{\omega^2} + 1 \bigg).$$
Because $u < F_{\eps}^{-1}(u) < 2u$, we have $$\log \bigg( \frac{F_{\eps}^{-1}(u)^2}{\omega^2} + 1 \bigg) = 2 \log^+ u + O(1).$$  Also, $1/F'(v) = 1 + \omega^2/v^2 = 1 + O(u^{-2})$.  Therefore,
$$\eps^{-1/2}( F_{\eps}^{-1}(u) - F^{-1}(u)) = \frac{3}{4 \sqrt{2}} (1 + O(u^{-2}))(2 \log^+ u + O(1)) = \frac{3}{2 \sqrt{2}} \log^+(u) + O(1),$$
which proves the result when $u > u^*$.
\qedwhite
\\

Equation \eqref{Fcompare} implies that for all $t \geq 0$, we have
\begin{equation}\label{L>L*}
L_\eps(t)>L_\eps^*(t).
\end{equation}
Also, by \eqref{Lbardef} and Lemma \ref{LepsLbar}, for $0 \leq t\lesssim\eps^{-3/2}$, we have
\begin{equation}\label{LL*1}
 L_\eps(t)-L_\eps^*(t)\lesssim 1.
\end{equation}
We next compute the first and second order derivatives of $L^*_\eps(t)$ and $L_\eps(t)$. We have
\begin{equation}\label{L'}
(L_\eps^*)'(t)=\eps \bigg(1+\frac{\omega^2}{F^{-1}(\eps^{3/2}t)^2}\bigg),
\end{equation}
\begin{equation}\label{L''}
(L_\eps^*)''(t)=-2\omega^2\eps^{5/2}\frac{F^{-1}(\eps^{3/2}t)^2+\omega^2}{F^{-1}(\eps^{3/2}t)^5}.
\end{equation}
Therefore, $L_\eps^*(t)$ is an increasing concave function on $(0,\infty)$ for every $\eps>0$. By the Mean Value Theorem, we have for all $\eps>0$ and $0<s<t$,
\begin{equation}\label{mvt}
0<(L_\eps^*)'(t)(t-s)\leq L_\eps^*(t)-L_\eps^*(s)\leq (L_\eps^*)'(s)(t-s).
\end{equation}
For $L_\eps(t)$, we get for $t > 0$,
\begin{equation}\label{newL'}
L_\eps'(t)=\frac{2\sqrt{2}\eps(\omega^2+F^{-1}_\eps(\eps^{3/2}t)^2)}{2\sqrt{2}F_\eps^{-1}(\eps^{3/2}t)^2-3\eps^{1/2}F_\eps^{-1}(\eps^{3/2}t)},
\end{equation}
\begin{equation}\label{newL''}
L_\eps''(t)=\frac{8\eps^{5/2}(\omega^2+F_\eps^{-1}(\eps^{3/2}t)^2)(-3\eps^{1/2}F_\eps^{-1}(\eps^{3/2}t)^2-4\sqrt{2}\omega^2F_\eps^{-1}(\eps^{3/2}t)+3\omega^2\eps^{1/2})}{(2\sqrt{2}F_\eps^{-1}(\eps^{3/2}t)^2-3\eps^{1/2}F_\eps^{-1}(\eps^{3/2}t))^3}.
\end{equation}

Given $0 < r < s$ and a nonnegative function $f$ defined on an interval containing $(r,s)$, we define
$$\tau^f(r,s) =\int_r^s \frac{1}{f(u)^2}du.$$
In the next two lemmas, we compute $\tau^{K_\eps}(r,s)$ and $\tau^{H_\eps}(r,s)$.

\begin{Lemma}\label{tau(r,s)}
For $0\leq r<s\leq t$, we have
\[
\tau^{K_\eps}(r,s)=\frac{1}{\omega^2}\big(K_\eps(r)-K_\eps(s)-\eps(s-r)\big).
\]
\end{Lemma}

\noindent\textit{Proof.} After the change of variable ~$z=F^{-1}(\eps^{3/2}(t-u))$, we get
\begin{align*}
\tau^{K_\eps}(r,s)
&=\int_{r}^{s}\frac{1}{K_\eps(u)^2}du\\
&= \eps^{-1/2}\int_{F^{-1}(\eps^{3/2}(t-s))}^{F^{-1}(\eps^{3/2}(t-r))}\frac{1}{\omega^2+z^2}dz\nonumber\\
&=\frac{1}{\omega}\eps^{-1/2}\bigg(\arctan\bigg(\frac{F^{-1}(\eps^{3/2}(t-r))}{\omega}\bigg)-\arctan\bigg(\frac{F^{-1}(\eps^{3/2}(t-s))}{\omega}\bigg)\bigg).
\end{align*}
Note that
\[
u=F\big(F^{-1}(u)\big)=F^{-1}(u)-\omega\arctan\bigg(\frac{F^{-1}(u)}{\omega}\bigg).
\]
The lemma follows from the previous two equations.
\qedwhite

\begin{Lemma}\label{newtau(r,s)}
For $0\leq r<s\leq t$, we have
\[
\tau^{H_\eps}(r,s)=\frac{1}{\omega^2}\big(H_\eps(r)-H_\eps(s)-\eps(s-r)\big)-\frac{3}{\pi^2}\log\left(\frac{H_\eps(r)}{H_\eps(s)}\right).
\]
\end{Lemma}

\begin{proof}
By equation \eqref{newL'}, we observe that for $0 \leq u \leq t$, we have
\begin{align*}
  H_\eps'(u)  &= - \frac{2 \sqrt{2}\eps (\omega^2 + F^{-1}_\eps(\eps^{3/2}(t-u))^2)}{2 \sqrt{2} F^{-1}_\eps(\eps^{3/2}(t-u))^2 - 3 \eps^{1/2} F^{-1}_\eps(\eps^{3/2}(t-u))} = - \frac{2 \sqrt{2} (\omega^2 + \eps H_\eps(u)^2)}{2\sqrt{2} H_\eps(u)^2 - 3 H_\eps(u)},
\end{align*}
using that $F^{-1}_\eps(\eps^{3/2}(t-u)) = \eps^{1/2} H_\eps(u)$. As a result, for all $0 \leq u \leq t$, we have
\[
\frac{1}{H_\eps(u)^2}=-\frac{1}{\omega^2}H_\eps'(u)-\frac{1}{\omega^2}\eps+\frac{3}{2\sqrt{2}\omega^2}\frac{H'_\eps(u)}{H_\eps(u)}.
\]
Integrating $u$ from $r$ to $s$, the lemma follows, using that $\frac{3}{2\sqrt{2}\omega^2} = \frac{3}{\pi^2}$.
\end{proof}

\subsection{Estimating the density when particles are killed at $K_{\eps}(\cdot)$}

Let $f: [0, t)\to [0,\infty)$ be some positive smooth curve. The function $f$ could be either $H_\eps$ or $K_\eps$. Consider a BBM process with drift $-(\sqrt{2}+\eps)$ in which particles are killed when they reach either the origin or the right boundary $f(s)$ at some time $s$. We denote by $E^{f}_{r,x}$ the expectation for the process which starts from a single particle at $x\in (0,f(r))$ at time $r$. For simplicity, when $r=0$, we write $E_{0,x}^{f}$ as $E_x^{f}$. We denote by $P_{r,x}^{f}$ and $P_x^{f}$ the corresponding probability measures. For $0\leq r<s\leq t$, define $q^f_{r,s}(x,y)$ to be the density of particles at time $s$ that are descendants of a particle at the location $x$ at
time $r$. To be more precise, for $U$ a Borel subset of $(0,f(s))$, the expected number of particles in $U$ at time $s$ which are descendants of a particles at $x$ at time $r$ is
\[
\int_{U}q^f_{r,s}(x,y)dy.
\]
Let $p^f_{r,s}(x,y)$ be the density for the position of a single Brownian particle, without drift, at time $s$ when the particle starts from $x$ at time $r$ and is killed upon hitting either the origin or the curve $f(\cdot)$. By the many-to-one lemma and Girsanov's theorem, the density $q^f_{r,s}(x,y)$ can be computed from $p^f_{r,s}(x,y)$. The proof of Proposition 5.4 in \cite{ms20} gives us ways to estimate $p^f_{r,s}(x,y)$ and therefore $q^f_{r,s}(x,y)$. A key ingredient is the following lemma, which is Lemma 5.3 in \cite{ms20} and was derived there from earlier work in \cite{nov, Roberts}. Define
\begin{equation}\label{densityw}
\omega_s(x,y) =2\sum_{n=1}^{\infty} e^{-\pi^2n^2s/2}\sin (n\pi x)\sin (n\pi y).
\end{equation}

\begin{Lemma}\label{lem5.3}
Let $0\leq r < s \leq t$. Under $E_{r,x}^{f}$, we denote by $\{B_u\}_{r\leq u\leq t}$ the trajectory of the Brownian particle started from $x$ at time $r$ when the particle is killed if it reaches either 0 or $f (u)$ at time $u$. Define
\[
\phi^f_{r,s} = \left(\frac{f(r)}{f (s)}\right)^{1/2}\exp\left(\frac{f'(s)B_s^2}{2f(s)}-\frac{f'(r)B_r^2}{2f(r)}-\int_r^s \frac{f''(u)B_u^2}{2f(u)}du\right).
\]
For any bounded measurable function $g: [0,f(s)]\rightarrow \R$, we have
\[
E_{r,x}^{f}\left[\phi^f_{r,s} g(B_s)\right]=\frac{1}{f(s)}\int_0^{f(s)}g(y)\omega_{\tau^f(r,s)}\left(\frac{x}{f(r)}, \frac{y}{f(s)}\right)dy.
\]
\end{Lemma}

In the rest of this subsection, we use Lemma \ref{lem5.3} to estimate $q_{r,s}^{K_{\eps}}(x,y)$.  We estimate $q_{r,s}^{H_{\eps}}(x,y)$ in subsection \ref{qHsub}.

\begin{Lemma}\label{density}
For $0\leq r < s < t$, $x\in (0, K_\eps(r))$ and $y\in (0, K_\eps (s))$, we have for all $\eps>0$,
\begin{equation}\label{densityBrownian}
p_{r,s}^{K_\eps}(x,y)=\frac{1}{K_\eps(r)^{1/2}K_\eps(s)^{1/2}}e^{O(\eps K_\eps(r)+1/K_\eps(s))}\omega_{\tau^{K_\eps}(r,s)}\left(\frac{x}{K_\eps(r)}, \frac{y}{K_\eps(s)}\right),
\end{equation}
and
\begin{equation}\label{densitygeneral}
q_{r,s}^{K_\eps}(x,y)=\frac{1}{K_\eps(r)^{1/2}K_\eps(s)^{1/2}}e^{\rho(x-y)-(\sqrt{2}\eps+\eps^2/2)(s-r)+O(\eps K_\eps(r)+1/K_\eps(s))}\omega_{\tau^{K_\eps}(r,s)}\left(\frac{x}{K_\eps(r)}, \frac{y}{K_\eps(s)}\right).
\end{equation}
In particular, if (\ref{limtz}) holds and $ s \leq t-\newC\label{C_a}$ for some positive constant $\oldC{C_a}$, then for $0\leq r<s$, $x\in (0, K_\eps(r))$ and $y\in (0, K_\eps (s))$, we have
\begin{align}\label{densityulbs}
 q_{r,s}^{K_\eps}(x,y)
\asymp \frac{1}{K_\eps(r)^{1/2}K_\eps(s)^{1/2}}e^{\rho(x-y)-\sqrt{2}\eps(s-r)}\omega_{\tau^{K_\eps}(r,s)}\left(\frac{x}{K_\eps(r)}, \frac{y}{K_\eps(s)}\right).
\end{align}
\end{Lemma}

\noindent{\it Proof.}
First, we apply Lemma \ref{lem5.3} to estimate $p_{r,s}^{K_\eps}(x,y)$. The key is to approximate ~$\phi_{r,s}^{K_\eps}$. Under $E_{r,x}^{K_\eps}$, we note that $0<B_u<K_\eps(u)$ for all $u\in [r,s]$. By equations (\ref{L'}) and (\ref{L''}), we have that under $E_{r,x}^{K_\eps}$,
\begin{align}\label{rho_rs1}
&\bigg|\frac{K'_\eps(s)B_s^2}{2K_\eps(s)}-\frac{K_\eps'(r)B_r^2}{2K_\eps(r)}-\int_r^s \frac{K_\eps''(u)B_u^2}{2K_\eps(u)}du\bigg| \nonumber\\
&\quad \leq \bigg|\frac{K_\eps'(s)K_\eps(s)}{2}\bigg|+\bigg|\frac{K_\eps'(r)K_\eps(r)}{2}\bigg|+\int_r^s \frac{|K_\eps''(u)|K_\eps(u)}{2}du \nonumber\\
&\quad =\frac{1}{2}\bigg(\eps K_\eps(s)+\frac{\omega^2}{K_\eps(s)}\bigg)+\frac{1}{2}\bigg(\eps K_\eps(r)+\frac{\omega^2}{K_\eps(r)}\bigg)+\omega^2\eps^2\int_r^s \frac{F^{-1}(\eps^{3/2}(t-u))^2+\omega^2}{F^{-1}(\eps^{3/2}(t-u))^4}du.
\end{align}
After the change of variable $v=F^{-1}(\eps^{3/2}(t-u))$, we get
\begin{equation}\label{rho_rs2}
\omega^2\eps^2\int_r^s \frac{F^{-1}(\eps^{3/2}(t-u))^2+\omega^2}{F^{-1}(\eps^{3/2}(t-u))^4}du=\omega^2\eps^{1/2}\int_{F^{-1}(\eps^{3/2}(t-s))}^{F^{-1}(\eps^{3/2}(t-r))}\frac{1}{v^2}dv=\omega^2\bigg(\frac{1}{K_\eps(s)}-\frac{1}{K_\eps(r)}\bigg).
\end{equation}
Equations (\ref{rho_rs1}) and (\ref{rho_rs2}) imply that for all $0\leq r<s<t$
\begin{equation}\label{rho_rs}
\bigg|\frac{K'_\eps(s)B_s^2}{2K_\eps(s)}-\frac{K_\eps'(r)B_r^2}{2K_\eps(r)}-\int_r^s \frac{K_\eps''(u)B_u^2}{2K_\eps(u)}du\bigg|\lesssim \eps K_\eps(r) +\frac{1}{K_\eps(s)}.
\end{equation}
Equation (\ref{densityBrownian}) follows from Lemma \ref{lem5.3} and (\ref{rho_rs}).

Next, by the many-to-one lemma and the Girsanov's theorem, we have
\begin{equation}\label{manygir}
q_{r,s}^{K_\eps}(x,y)=e^{s-r}e^{-\rho^2(s-r)/2+\rho(x-y)}p^{K_\eps}_{r,s}(x,y)=e^{\rho(x-y)-(\sqrt{2}\eps+\eps^2/2)(s-r)}p_{r,s}^{K_\eps}(x,y).
\end{equation}
Here, the factor $e^{s-r}$ represents the expected number of particles at time $s$ if there is no killing and the process starts from a single particle at time $r$, and $e^{-\rho^2(s-r)/2+\rho(x-y)}$ is the Girsanov factor which transforms Brownian motion with drift $-\rho$ to ordinary Brownian motion. Equation~(\ref{densitygeneral}) follows from (\ref{densityBrownian}) and (\ref{manygir}).

In particular, if $t\lesssim \eps^{-3/2}$ and $0\leq r < s \leq t-\oldC{C_a}$ for some positive constant $\oldC{C_a}$, then by (\ref{F-1asymp3/5}), we have for $\eps$ sufficiently small,
\begin{equation}\label{1/K}
\frac{1}{K_\eps(s)}\leq \frac{1}{\eps^{-1/2}F^{-1}(\eps^{3/2}\oldC{C_a})}\lesssim 1.
\end{equation}
Moreover, by (\ref{L_eps<2epst}) and  (\ref{L>L*}), we have
\begin{equation}\label{epsK}
\eps K_\eps(r)\leq \eps L_\eps(t)\lesssim 1.
\end{equation} Therefore, we have $|-\eps^2(s-r)/2+O(\eps K_\eps(r)+1/K_\eps(s))|\lesssim 1$ and (\ref{densityulbs}) follows from~(\ref{densitygeneral}).
\qedwhite
\\

For some choices of $r,s$ and $t$, the infinite sum expression for $q_{r,s}(x,y)$ can be simplified. Define
\[
J_t=\sum_{n=2}^{\infty} n^2e^{-\pi^2(n^2-1)t/2}.
\]

\begin{Lemma}\label{densityv}
Suppose (\ref{limtz}) holds. If $r+\newC\label{C_b}K_\eps(r)^2\leq s$ for some constants $\oldC{C_b}$
 and $s\leq t- \oldC{C_a}$, then for $x\in (0, K_\eps(r))$ and $y\in (0, K_\eps(s))$, we have
\[
 q_{r,s}^{K_\eps}(x,y)
 \asymp \frac{1}{K_{\eps}(r)^{1/2}K_{\eps}(s)^{1/2}}e^{\rho (x-y)-\sqrt{2}(K_\eps(r)-K_\eps(s))}\sin\left(\frac{\pi x}{K_{\eps}(r)}\right)\sin\left(\frac{\pi y}{K_{\eps}(s)}\right).
\]
\end{Lemma}

\noindent\textit{Proof.} According to Lemma 5 in \cite{bbs13}, for all $0\leq r<s$, $x\in (0, K_\eps(r))$ and $y\in (0,K_\eps(s))$, we have
\begin{equation}\label{bbs13lem5}
\omega_{\tau^{K_\eps}(r,s)}\left(\frac{x}{K_\eps(r)}, \frac{y}{K_\eps(s)}\right)=2e^{-\pi^2\tau^{K_\eps}(r,s)/2}\sin\left(\frac{\pi x}{K_{\eps}(r)}\right)\sin\left(\frac{\pi y}{K_{\eps}(s)}\right)(1+D_{r,s}(x,y))
\end{equation}
where
\[
|D_{r,s}(x,y)|\leq J_{\tau^{K_\eps}(r,s)}.
\]
For $s\geq r+\oldC{C_b}K_\eps(r)^2$, we observe that
\begin{equation}\label{taulb}
\tau^{K_\eps}(r,s)=\int_r^s\frac{1}{K_\eps(u)^2}du\geq \frac{s-r}{K_\eps(r)^2}\geq \oldC{C_b},
\end{equation}
which implies that $J_{\tau^{K_\eps(r,s)}}\leq J_{\oldC{C_b}}$. Lemma \ref{densityv} follows from equations (\ref{densityulbs}), (\ref{bbs13lem5}), and (\ref{taulb}), and Lemma \ref{tau(r,s)}.
\qedwhite

\begin{Rmk}
Note that the density formula in Lemma \ref{densityv} is similar to the critical case. Define $K_0(s)=c(t-s)^{1/3}$. Consider the critical process where particles are killed when they hit either $0$ or $K_0(s)$ at time $s$. Let $q_{r,s}^{K_0}(x,y)$ be the corresponding density. According to Proposition 12 in ~\cite{bbs14} or Proposition 5.4 in \cite{ms20}, there exists a constant $\newC\label{C_c}$ such that for $r+K_0(r)^2\leq s\leq t-\oldC{C_c}$,
\[
q^{K_0}_{r,s}(x,y) \asymp\frac{1}{K_0(r)^{1/2}K_0(s)^{1/2}}e^{\sqrt{2}(x-y)-\sqrt{2}(K_0(r)-K_0(s))}\sin\left(\frac{\pi x}{K_0(r)}\right)\sin\left(\frac{\pi y}{K_0(s)}\right).
\]
\end{Rmk}

\subsection{Estimating the density when particles are killed at $H_{\eps}(\cdot)$}\label{qHsub}

\begin{Lemma}\label{densityH}
Suppose $t>0$. For $0\leq r < s \leq t$, $x\in (0, H_\eps(r))$ and $y\in (0, H_\eps (s))$, we have for all $\eps>0$,
\begin{equation}\label{densityBrownianH}
p_{r,s}^{H_\eps}(x,y)
=\frac{1}{H_\eps(r)^{1/2}H_\eps(s)^{1/2}}e^{O(\eps H_\eps(r)+1/H_\eps(s)+\eps|\log(\eps^{1/2}H_\eps(r))|)}\omega_{\tau^{H_\eps}(r,s)}\left(\frac{x}{H_\eps(r)}, \frac{y}{H_\eps(s)}\right).
\end{equation}
In particular, if (\ref{tles}) holds, then for $0\leq r< s \leq t$, $x\in (0, H_\eps(r))$ and $y\in (0, H_\eps (s))$, we have
\begin{align}\label{densityulbsH}
 q_{r,s}^{H_\eps}(x,y)
\asymp \frac{1}{H_\eps(r)^{1/2}H_\eps(s)^{1/2}}e^{\rho(x-y)-\sqrt{2}\eps(s-r)}\omega_{\tau^{H_\eps}(r,s)}\left(\frac{x}{H_\eps(r)}, \frac{y}{H_\eps(s)}\right).
\end{align}
\end{Lemma}

\noindent\textit{Proof.} The proof is essentially the same as the proof of Lemma \ref{density}. The main difference is the approximation of $\phi^{H_\eps}_{r,s}$.  Under $E_{r,x}^{H_\eps}$,
\begin{align}\label{rho_rs1H}
&\bigg|\frac{H'_\eps(s)B_s^2}{2H_\eps(s)}-\frac{H_\eps'(r)B_r^2}{2H_\eps(r)}-\int_r^s \frac{H_\eps''(u)B_u^2}{2H_\eps(u)}du\bigg| \nonumber\\
&\hspace{0.3in} \leq \bigg|\frac{H_\eps'(s)H_\eps(s)}{2}\bigg|+\bigg|\frac{H_\eps'(r)H_\eps(r)}{2}\bigg|+\int_r^s \frac{|H_\eps''(u)|H_\eps(u)}{2}du.
\end{align}
We denote by $I_1$, $I_2$ and $I_3$ the three terms on the right hand side of equation (\ref{rho_rs1H}). By equation (\ref{newL'}), we see that
\begin{align}\label{densityI1int}
I_1
&=\frac{\sqrt{2}\eps^{1/2}(\omega^2+F^{-1}_\eps(\eps^{3/2}(t-s))^2)}{2\sqrt{2}F_\eps^{-1}(\eps^{3/2}(t-s))-3\eps^{1/2}}.
\end{align}
Note that for $s\leq t$, by \eqref{Leps0}, we have
\begin{equation}\label{newFinv9pi}
F_\eps^{-1}(\eps^{3/2}(t-s))\geq F_{\eps}^{-1}(0) \geq \frac{9}{4 \sqrt{2}}\eps^{1/2}.
\end{equation}
Putting together (\ref{densityI1int}) and (\ref{newFinv9pi}), we have
\begin{equation}\label{densityI1}
I_1\leq \frac{\eps^{1/2}(\omega^2+F^{-1}_\eps(\eps^{3/2}(t-s))^2)}{F_\eps^{-1}(\eps^{3/2}(t-s))}\lesssim \frac{1}{H_\eps(s)}+\eps H_\eps(s).
\end{equation}
Following the same argument, we get
\begin{equation}\label{densityI2}
I_{2} \lesssim \frac{1}{H_\eps(r)}+\eps H_\eps(r).
\end{equation}
To compute $I_3$, we apply equation (\ref{newL''}). Note that for $0\leq r\leq u\leq s\leq t$, the denominator in~(\ref{newL''}) is always positive by (\ref{newFinv9pi}) and thus $|H''_\eps(u)|$ is bounded above by
\[
\frac{8\eps^{5/2}(\omega^2+F_\eps^{-1}(\eps^{3/2}(t-u))^2)(3\eps^{1/2}F_\eps^{-1}(\eps^{3/2}(t-u))^2+4\sqrt{2}\omega^2F_\eps^{-1}(\eps^{3/2}(t-u))+3\omega^2\eps^{1/2})}{(2\sqrt{2}F_\eps^{-1}(\eps^{3/2}(t-u))^2-3\eps^{1/2}F_\eps^{-1}(\eps^{3/2}(t-u)))^3}.
\]
By (\ref{newFinv9pi}), after making the change of variable $v=F^{-1}_\eps(\eps^{3/2}(t-u))$, for $0 \leq u \leq t$ we have $v \geq \frac{9}{4 \sqrt{2}} \eps^{1/2}$ and therefore $2 \sqrt{2}v - 3 \eps^{1/2} \geq \frac{2 \sqrt{2}}{3} v$.  It follows that for $0\leq r\leq u\leq s\leq t$,
\begin{align}\label{densityI3}
I_3
&= \sqrt{2}\eps^{1/2}\int_{F^{-1}_\eps(\eps^{3/2}(t-s))}^{F^{-1}_\eps(\eps^{3/2}(t-r))} \frac{3\eps^{1/2}v^2+4\sqrt{2}\omega^2 v+3\omega^2 \eps^{1/2}}{v(2\sqrt{2}v-3\eps^{1/2})^2}dv\nonumber\\
&\lesssim \eps \int_{F^{-1}_\eps(\eps^{3/2}(t-s))}^{F^{-1}_\eps(\eps^{3/2}(t-r))} \frac{1}{v} \: dv + \eps^{1/2} \int_{F^{-1}_\eps(\eps^{3/2}(t-s))}^{F^{-1}_\eps(\eps^{3/2}(t-r))} \frac{v + \eps^{1/2}}{v^3} \: dv \nonumber \\
&\lesssim \eps \left|\log\left(\eps^{1/2}H_\eps(r)\right)\right|+\frac{1}{H_\eps(s)}.
\end{align}
Equation (\ref{densityBrownianH}) follows from Lemma \ref{lem5.3} and equations (\ref{rho_rs1H}), (\ref{densityI1})-(\ref{densityI3}).

Applying the many-to-one lemma and the Girsanov's theorem, it follows from (\ref{densityBrownianH}) that
\begin{align}\label{densitygeneralH}
q_{r,s}^{H_\eps}(x,y)
&=\frac{1}{H_\eps(r)^{1/2}H_\eps(s)^{1/2}}e^{\rho(x-y)-(\sqrt{2}\eps+\eps^2/2)(s-r)+O(\eps H_\eps(r)+1/H_\eps(s)+\eps|\log(\eps^{1/2}H_\eps(r))|)}\nonumber\\
&\hspace{0.2in}\times\omega_{\tau^{H_\eps}(r,s)}\left(\frac{x}{H_\eps(r)}, \frac{y}{H_\eps(s)}\right).
\end{align}
In particular, if (\ref{tles}) holds, then for $0\leq r< s \leq t$, by (\ref{L_eps<2epst}) and  (\ref{epsK}), we have $\eps H_\eps(r)\lesssim 1$ and $\eps\log(\eps^{1/2}H_\eps(r))\ll 1$. We also have $1/H_\eps(s)\lesssim 1$ by (\ref{newFinv9pi}).  Therefore, we have $$\left|-\frac{\eps^2(s-r)}{2}+O\left(\eps H_\eps(r)+\frac{1}{H_\eps(s)}+\eps \left|\log(\eps^{1/2}H_\eps(r))\right|\right) \right|\lesssim 1$$ and equation (\ref{densityulbsH}) follows from~(\ref{densitygeneralH}).
\qedwhite
\\

\begin{Lemma}\label{densityvH}
Suppose (\ref{tles}) holds and $t>0$. If $r+\oldC{C_b}H_\eps(r)^2\leq s$ for some constant $\oldC{C_b}$ and $s\leq t$, then for $x\in (0, H_\eps(r))$ and $y\in (0, H_\eps(s))$, we have
\[
 q_{r,s}^{H_\eps}(x,y)
 \asymp \frac{H_\eps(r)}{H_{\eps}(s)^{2}}e^{\rho (x-y)-\sqrt{2}(H_\eps(r)-H_\eps(s))}\sin\left(\frac{\pi x}{H_{\eps}(r)}\right)\sin\left(\frac{\pi y}{H_{\eps}(s)}\right).
\]
\end{Lemma}
\noindent\textit{Proof.} Applying equation (\ref{densityulbsH}) in place of equation (\ref{densityulbs}) and Lemma \ref{newtau(r,s)} in place of Lemma \ref{tau(r,s)}, the proof is a word-by-word repetition of the proof of Lemma \ref{densityv}.
\qedwhite

\section{Number of particles hitting the right boundary}\label{numberhitting}

Consider a process in which particles are killed when they reach either the origin or the right boundary $f(\cdot)$. For $0\leq r<s\leq t$, suppose $f$ is a smooth function in the interval $[0,s]$. Let $R_{t}^f(r,s)$ be the number of particles that are killed at the right boundary $f(\cdot)$ during the time interval $[r,s]$. This section aims to give first and second moment estimates of $R_t^{K_\eps}(r,s)$ and $R_t^{H_\eps}(r,s)$ following the argument of Lemmas 5.8 and 5.10 in \cite{ms20} and Lemma 16 in \cite{bbs14}. A key input is the following lemma which shows that the rate at which the Brownian particle hits the right boundary $f(\cdot)$ can be computed from the derivative of the density at the right boundary. This result is well-known in the literature. Lemma \ref{lem5.7} is adapted from Lemma 5.7 in \cite{ms20}.

\begin{Lemma}\label{lem5.7}
Under $P_{r,x}^{f}$, we denote by $\{B_u\}_{r\leq u\leq s}$ the trajectory of the Brownian particle started from $x$ at time $r$ when the particle is killed if it reaches either 0 or $f(u)$ at time $u$. Let $\tau^+$ and $\tau^-$ be the hitting times of the right boundary $f(\cdot)$ and the origin respectively. Then for $r\leq u\leq s$, we have
\begin{equation*}
P_{r,x}^{f}(\tau^+<\tau^-, \tau^+\in du)=-\frac{1}{2}\frac{\partial}{\partial y} p_{r,u}^f(x,y)\bigg|_{y=f(u)}du.
\end{equation*}
\end{Lemma}

We need two more results which compute integrals involving $\omega_u(x,y)$. For a measurable set $S\subset \R_+$, define
\[
I(x, S)=\int_S e^{\pi^2 u/2}\left(-\frac{1}{2}\frac{\partial}{\partial y}\omega_u (x,y)\bigg|_{y=1}\right)du.
\]
We denote by $\lambda(S)$ the Lebesgue measure of $S$. Lemma 7.1 in \cite{m16} states that there exists a universal constant $\newC\label{C_mai}$ such that for every $x\in [0,1]$ and every measurable set $S\subset \R_+$,
\begin{equation}\label{lem7.1}
|I(x,S)-\pi\lambda(S)\sin (\pi x)|\leq \oldC{C_mai}\min\left\{ x, J_{\inf S}\sin(\pi x)\min\{1, \lambda(S)\}\right\}.
\end{equation}
The following lemma is adapted from Lemma 5.1 and 5.2 of \cite{ms20}.
\begin{Lemma}\label{lem5.1/2}
For all $x\in (0,1)$ and $y\in (0,1/2]$, we have
\begin{equation}\label{lem5.1small}
\int_0^1 e^{\pi^2 u/2}\sup_{y'\in [0,y]}\omega_u(x,y')du=O\left(y(1-x)\right).
\end{equation}
If $s>\oldC{C_2ndH}$ for some positive constant $\newC\label{C_2ndH}$, then for all $x\in (0,1)$ and $y\in (0,1/2]$,
\begin{equation}\label{lem5.1s}
\int_{\oldC{C_2ndH}}^s e^{\pi^2 u/2}\sup_{y'\in [0,y]}\omega_u(x,y')du=O\left(ys\sin(\pi x)\right).
\end{equation}
For all $x\in (0,1)$, we have
\begin{equation}\label{lem5.2}
\int_0^s e^{\pi^2u/2}\int_0^1 \omega_u(x,y)dydu=O(s\sin(\pi x)+(1-x)).
\end{equation}
\end{Lemma}
\noindent\textit{Proof.} Equation (\ref{lem5.1small}) follows from equation (5.9) in \cite{ms20}. Equation (\ref{lem5.1s}) follows from equation (5.8) in \cite{ms20}. Equation (\ref{lem5.2}) is Lemma 5.2 in \cite{ms20}.
\qedwhite
\\

\subsection{Moment estimates of $R_t^{K_\eps}(r,s)$}
\begin{Lemma}\label{lem15}
Suppose (\ref{limtz}) holds and $s \leq t-\oldC{C_a}$ for some positive constant $\oldC{C_a}$. Then for $0\leq v\leq r<s$ and $x\in (0, K_\eps(v))$, we have
\begin{equation*}
E_{v,x}^{K_{\eps}}\left[R_{t}^{K_\eps}(r,s)\right]\lesssim e^{\rho(x-K_\eps(v))}\left(\tau^{K_\eps}(r,s)\sin\left(\frac{\pi x}{K_{\eps}(v)}\right)+\frac{x}{K_\eps(v)}\right).
\end{equation*}
\end{Lemma}

\noindent\textit{Proof.}
By the many-to-one lemma, Girsanov's theorem and Lemma \ref{lem5.7}, we have
\begin{align}\label{manygirhit}
E_{v,x}^{K_\eps}\left[R_t^{K_\eps}(r,s)\right]
&=\int_r^s e^{u-v}e^{-\rho^2(u-v)/2+\rho(x-K_\eps(u))}P_{v,x}^{K_\eps}(\tau^+<\tau^-, \tau^+\in du)\nonumber\\
&=\int_r^s e^{u-v}e^{-\rho^2(u-v)/2+\rho(x-K_\eps(u))}\left(-\frac{1}{2}\frac{\partial}{\partial y} p^{K_\eps}_{v,u}(x,y)\bigg|_{y=K_\eps(u)}\right)du.
\end{align}
Since the estimate of $p^{K_\eps}_{v,u}(x,y)$ in (\ref{densityBrownian}) holds uniformly for all $0\leq v<s<t$, and $y\in (0, K_\eps(s))$ and the derivative of $p_{v,u}^{K_\eps}(x,y)$ with respect to $y$ exists for all $u\in [r,s]$, it follows from Lemma \ref{tau(r,s)} and equations (\ref{densityBrownian}) and (\ref{manygirhit}) that
\begin{align}\label{lem15intstep}
E_{v,x}^{K_\eps}\left[R_t^{K_\eps}(r,s)\right]
&=e^{O(\eps K_\eps(v)+1/K_\eps(s))}\int_r^s\frac{1}{K_\eps(v)^{1/2}K_\eps(u)^{3/2}}e^{-\eps^2(u-v)/2+\eps(K_\eps(v)-K_\eps(u))+\rho(x-K_\eps(v))}\nonumber\\
&\hspace{0.2in}\times e^{\pi^2\tau^{K_\eps}(v,u)/2}\bigg(-\frac{1}{2}\frac{\partial}{\partial y}\omega_{\tau^{K_\eps}(v,u)}\left(\frac{x}{K_\eps(v)}, y\right)\bigg|_{y=1}\bigg)du.
\end{align}
If $t\lesssim \eps^{-3/2}$ and $0\leq v \leq r<s \leq t-\oldC{C_a}$ for some positive constant $\oldC{C_a}$, then $$\left|-\frac{\eps^2(u-v)}{2}+\eps(K_\eps(v)-K_\eps(u))+O\left(\eps K_\eps(v)+\frac{1}{K_\eps(s)}\right)\right|\lesssim 1$$ for all $u\in [r,s]$  by (\ref{1/K}) and (\ref{epsK}). Therefore,
we have for $\eps$ sufficiently small
\begin{align}\label{EvxR(r,s)int}
E_{v,x}^{K_\eps}\left[R_t^{K_\eps}(r,s)\right]
&\asymp e^{\rho(x-K_\eps(v))}\int_r^s\frac{1}{K_\eps(v)^{1/2}K_\eps(u)^{3/2}} e^{\pi^2\tau^{K_\eps}(v,u)/2}\nonumber\\
&\qquad \qquad\times\bigg(-\frac{1}{2}\frac{\partial}{\partial y}\omega_{\tau^{K_\eps}(v,u)}\left(\frac{x}{K_\eps(v)}, y\right)\bigg|_{y=1}\bigg)du.
\end{align}
Note that $K_\eps(v)\geq K_\eps(u)$. After the change of variable $z=\tau^{K_\eps}(v,u)$, by equation (\ref{lem7.1}), we obtain
\begin{align}\label{lem15proof}
E_{v,x}^{K_\eps}\left[R_t^{K_\eps}(r,s)\right]
&\lesssim e^{\rho(x-K_\eps(v))}\int_r^s\frac{1}{K_\eps(u)^{2}} e^{\pi^2\tau^{K_\eps}(v,u)/2}\bigg(-\frac{1}{2}\frac{\partial}{\partial y}\omega_{\tau^{K_\eps}(v,u)}\left(\frac{x}{K_\eps(v)}, y\right)\bigg|_{y=1}\bigg)du\nonumber\\
&=e^{\rho(x-K_\eps(v))}\int_{\tau^{K_\eps}(v,r)}^{\tau^{K_\eps}(v,s)}e^{\pi^2z/2}\bigg(-\frac{1}{2}\frac{\partial}{\partial y}\omega_{z}\left(\frac{x}{K_\eps(v)}, y\right)\bigg|_{y=1}\bigg)dv\nonumber\\
&\lesssim e^{\rho(x-K_\eps(v))}\bigg(\tau^{K_\eps}(r,s)\sin\bigg(\frac{\pi x}{K_\eps(v)}\bigg)+\frac{x}{K_\eps(v)}\bigg)
\end{align}
and the lemma follows.
\qedwhite
\\

Fix $0<\alpha<\beta<1$ and $\oldC{C_a}>0$. Choose $\oldC{C_b}$ large enough such that
\begin{equation}\label{constb}
J_{\oldC{C_b}}<\frac{1}{2}.
\end{equation}
Let $A_{\alpha,\beta}$ be a positive constant such that
\begin{equation}\label{Aalphabeta}
A_{\alpha,\beta}\geq\max\left\{\frac{\oldC{C_a}}{1-\beta}, \left(\frac{2c^2\oldC{C_b}}{\alpha}\right)^3, \left(\frac{2 \omega^2 \max\{\oldC{C_b}, \oldC{C_mai}\}}{c((1-\alpha)^{1/3}-(1-\beta)^{1/3})}\right)^3\right\}.
\end{equation}
We claim that for $t$ satisfying (\ref{limtz}) and for $t \geq A_{\alpha,\beta}$, if $\eps$ is sufficiently small, then
\begin{equation}\label{AalphabetaCa}
\alpha t\leq \beta t\leq t-\oldC{C_a},
\end{equation}
\begin{equation}\label{AalphabetaCb}
\oldC{C_b}K_\eps(0)^2\leq \alpha t,
\end{equation}
\begin{equation}\label{AalphabetaE}
\tau^{K_\eps}(\alpha t,\beta t) \geq \max\{\oldC{C_b}, \oldC{C_mai}\}\geq 2\oldC{C_mai} J_{\tau^{K_{\eps}}(0,\alpha t)}.
\end{equation}
Indeed, equation (\ref{AalphabetaCa}) is straightforward. When $t\gg 1/\eps$, equation (\ref{AalphabetaCb}) is obvious. When $t\lesssim 1/\eps$, because $A_{\alpha,\beta}\geq (2c^2\oldC{C_b})^3/\alpha$, equation (\ref{AalphabetaCb}) follows from (\ref{F-1asymp3/5}). To prove (\ref{AalphabetaE}), we first note that by (\ref{AalphabetaCb}),
\[
\tau^{K_{\eps}}(0,\alpha t)\geq \frac{\alpha t}{K_\eps(0)^2}\geq \oldC{C_b}.
\]
It follows from (\ref{constb})  that $2\oldC{C_mai}J_{\tau^{K_{\eps}}(0,\alpha t)}\leq \oldC{C_mai}$. When $t\asymp \eps^{-3/2}$, by Lemma \ref{tau(r,s)}, we have
\[
\frac{1}{\omega^2}K_\eps(\alpha t)\geq \tau^{K_\eps}(\alpha t,\beta t) \geq \frac{\beta t-\alpha t}{K_\eps(\alpha t)^2}=\frac{\beta-\alpha}{F^{-1}(\eps^{3/2}(1-\alpha)t)^2}\eps t.
\]
Thus
\begin{equation}\label{tauasymp2/3}
\tau^{K_\eps}(\alpha t,\beta t)\asymp \eps^{-1/2}
\end{equation} and (\ref{AalphabetaE}) holds trivially. When $t\ll \eps^{-3/2}$, according to equation (\ref{F-1asymp3/5}), Lemma \ref{tau(r,s)} and the fact that $t^{1/3}\gg \eps t$, we have
\begin{equation}\label{taull2/3}
\tau^{K_\eps}(\alpha t,\beta t)=\frac{c}{\omega^2}\big((1-\alpha)^{1/3}-(1-\beta)^{1/3}\big)t^{1/3}+O(\eps t)\geq \frac{c}{2\omega^2}\big((1-\alpha)^{1/3}-(1-\beta)^{1/3}\big)t^{1/3}
\end{equation}
and equation (\ref{AalphabetaE}) follows.
\\

\begin{Lemma}\label{lem15lb}
Suppose (\ref{limtz}) holds and $t\geq A_{\alpha,\beta}$. Then for $x\in (0,K_\eps(0))$,
\[
E_{x}^{K_{\eps}}\left[R_{t}^{K_\eps}(\alpha t,\beta t)\right]\asymp \tau^{K_\eps}(\alpha t,\beta t)\sin\left(\frac{\pi x}{K_{\eps}(0)}\right)e^{\rho(x-K_\eps(0))}.
\]
\end{Lemma}

\noindent\textit{Proof.}
We apply equation (\ref{EvxR(r,s)int}) with $v=0, r=\alpha t$ and $s=\beta t$. Since $K_\eps(\alpha t)\asymp K_\eps(u)\asymp K_\eps(\beta t)$ for all $u\in [\alpha t,\beta t]$, we get
\[
E_{x}^{K_\eps}\left[R_t^{K_\eps}(\alpha t,\beta t)\right]
\asymp e^{\rho(x-K_\eps(0))}\int_{\alpha t}^{\beta t}\frac{1}{K_\eps(u)^{2}} e^{\pi^2\tau^{K_\eps}(0,u)/2}\bigg(-\frac{1}{2}\frac{\partial}{\partial y}\omega_{\tau^{K_\eps}(0,u)}\left(\frac{x}{K_\eps(0)}, y\right)\bigg|_{y=1}\bigg)du.
\]
After the change of variable $z=\tau^{K_\eps}(0,u)$, by equation (\ref{lem7.1}), we obtain
\begin{equation}\label{lem15lb1}
E_{x}^{K_\eps}\left[R_t^{K_\eps}(\alpha t,\beta t)\right]
\lesssim e^{\rho(x-K_\eps(0))}\bigg(\pi\tau^{K_\eps}(\alpha t,\beta t)\sin\bigg(\frac{\pi x}{K_\eps(0)}\bigg)+\oldC{C_mai}J_{\tau^{K_\eps}(0,\alpha t)}\sin \left(\frac{\pi x}{K_\eps(0)}\right)\bigg)
\end{equation}
and
\begin{equation}\label{lem15lb2}
E_{x}^{K_\eps}\left[R_t^{K_\eps}(\alpha t,\beta t)\right]
\gtrsim e^{\rho(x-K_\eps(0))}\bigg(\pi\tau^{K_\eps}(\alpha t,\beta t)\sin\bigg(\frac{\pi x}{K_\eps(0)}\bigg)-\oldC{C_mai}J_{\tau^{K_\eps}(0,\alpha t)}\sin \left(\frac{\pi x}{K_\eps(0)}\right)\bigg).
\end{equation}
The lemma follows from (\ref{AalphabetaE}), (\ref{lem15lb1}) and (\ref{lem15lb2}).
\qedwhite
\\

\begin{Lemma}\label{lem16}
Suppose (\ref{limtz}) holds and $t\geq A_{\alpha,\beta}$. Then for $x\in (0,K_\eps(0)-1)$,
\[
E_{x}^{K_{\eps}}\left[R^{K_\eps}_{t}(\alpha t, \beta t)^2\right]\lesssim K_\eps(0)\sin\left(\frac{\pi x}{K_\eps(0)}\right)e^{\rho(x-K_\eps(0))}.
\]
\end{Lemma}

\noindent\textit{Proof.}
We use a standard second moment estimate.  It is noted, for example, in the proof of Lemma 16 in~\cite{bbs14} that we can write $R_t^{K_\eps}(u_1,u_2)^2=R_t^{K_\eps}(u_1,u_2)+2Y$, where $Y$ is the number of distinct pairs of particles that hit $H_\eps(\cdot)$ during time $(u_1,u_2)$. If a branching event happens at time $s$ and location $y$, then on average there will be $E_{s,y}^{K_\eps}[R_t^{K_\eps}(u_1 \vee s,u_2)]^2$ number of pairs of particles that hit $K_\eps(\cdot)$ during time $(u_1,u_2)$ and have their most recent common ancestor at time $s$.
Therefore, using Lemma \ref{lem15} and the fact that $\tau^{K_\eps}(s,\beta t)\leq K_\eps(s)/\omega^2$ for all $s\in [0,\beta t]$ by Lemma \ref{tau(r,s)}, we have
\begin{align}\label{R2I1234}
E_{x}^{K_{\eps}}\left[R_t^{K_\eps}(\alpha t, \beta t)^{2}\right]
&\leq E_{x}^{K_{\eps}}\left[R_t^{K_\eps}(\alpha t, \beta t)\right]+2\int_{0}^{\beta t}\int_{0}^{K_{\eps}(s)}q_{0,s}^{K_\eps}(x,y)\left(E_{s,y}^{K_{\eps}}\left[R_t^{K_\eps}(s, \beta t)\right]\right)^2dyds\nonumber\\
&\lesssim E_{x}^{K_{\eps}}\left[R_t^{K_\eps}(\alpha t, \beta t)\right]\nonumber\\
&\hspace{0.2in}+\int_{0}^{\beta t}\int_{0}^{K_{\eps}(s)}q_{0,s}^{K_\eps}(x,y)K_\eps(s)^2\sin^2\left(\frac{\pi y}{K_\eps(s)}\right)e^{2\rho(y-K_\eps(s))}dyds\nonumber\\
&\hspace{0.2in}+\int_{0}^{\beta t}\int_{0}^{K_{\eps}(s)}q_{0,s}^{K_\eps}(x,y)\frac{y^2}{K_\eps(s)^2}e^{2\rho(y-K_\eps(s))}dyds.
\end{align}
By Lemmas \ref{tau(r,s)} and \ref{lem15lb},
\begin{equation}\label{firstmomterm}
E_{x}^{K_{\eps}}\left[R_t^{K_\eps}(\alpha t, \beta t)\right] \lesssim K_\eps(0)\sin\left(\frac{\pi x}{K_\eps(0)}\right)e^{\rho(x-K_\eps(0))}.
\end{equation}
We divide the first double integral in (\ref{R2I1234}) into two pieces, denoted $I_1$ and $I_2$, depending on whether $0\leq s\leq \oldC{C_b}K_\eps(0)^2$ or $\oldC{C_b}K_\eps(0)^2< s\leq \beta t$. We do the same thing to the second double integral and get $I_3$ and $I_4$. We are going to bound $I_1$, $I_2$, $I_3$ and $I_4$ separately.

We start with $I_1$. By equations (\ref{L'}) and (\ref{mvt}), we have for $0\leq s\leq \oldC{C_b}K_{\eps}(0)^2$,
\begin{equation}\label{L(s)-L(r)}
K_\eps(0)-K_\eps(s)\leq \eps\oldC{C_b}K_\eps(0)^2\left(1+\frac{\omega^2}{F^{-1}(\eps^{3/2}(t-\oldC{C_b}K_\eps(0)^2))^2}\right).
\end{equation}
Note that $\eps \oldC{C_b} K_{\eps}(0)^2 = C_8 F^{-1}(\eps^{3/2} t)^2$.
Therefore, if $t\asymp \eps^{-3/2}$, then from (\ref{L(s)-L(r)}), we have
\begin{equation}\label{L0s}
0\leq K_{\eps}(0)-K_{\eps}(s)\lesssim 1.
\end{equation}
If $t\ll\eps^{-3/2}$, then by (\ref{AalphabetaCb}), we have $F^{-1}(\eps^{3/2}(t-\oldC{C_b}K_\eps(0)^2)) \geq F^{-1}(\eps^{3/2}(1 - \alpha)t)$, and then (\ref{F-1asymp3/5}) and (\ref{L(s)-L(r)}) imply that (\ref{L0s}) still holds. We thus have for all $0\leq s\leq \oldC{C_b}K_{\eps}(0)^2$ and $y\in (0,K_{\eps}(s))$,
\[
\sin\Big(\frac{\pi y}{K_{\eps}(s)}\Big)\lesssim \sin\Big(\frac{\pi y}{K_{\eps}(0)}\Big).
\]
Then by Fubini's theorem, we have
\begin{align}\label{lem15I1}
I_1
&\lesssim K_{\eps}(0)^2\int_{0}^{K_{\eps}(0)}\sin^2 \Big(\frac{\pi y}{K_{\eps}(0)}\Big)e^{2\rho (y-K_{\eps}(0))}\int_{0}^{\oldC{C_b}K_{\eps}(0)^2}q^{K_{\eps}}_{0,s}(x,y)dsdy.
\end{align}
We adapt Lemma 4 in \cite{bbs14} to obtain an upper bound for $\int_{0}^{\oldC{C_b}K_{\eps}(0)^2}q^{K_{\eps}}_{0,s}(x,y)ds$. Consider the process where particles move as Brownian motion with drift $-\rho$ and are killed upon hitting constant boundaries $0$ or $K_\eps(0)$. For $x\in (0,K_{\eps}(0))$, we denote by $q^*_s(x,y)$ the density of the process at time $s$ which starts from a single particle at $x$ at time $0$. Correspondingly, we denote by $v_s(x,y)$ the density of Brownian motion (without drift) in the strip $[0,K_\eps(0)]$ started from a single particle at $x$. By the many-to-one lemma, Girsanov's theorem and equation (51) in \cite{bbs13}, we have for all $s\geq 0$ and $x,y\in (0,K_{\eps}(0))$,
\begin{equation}\label{green}
\int_{0}^{\infty}q^*_s(x,y)ds=\int_0^\infty e^{\rho(x-y)-(\sqrt{2}\eps+\eps^2/2)s}v_s(x,y)ds\leq \frac{2e^{\rho(x-y)}x(K_{\eps}(0)-y)}{K_{\eps}(0)}.
\end{equation}
Note that $q_{0,s}(x,y)$ is bounded above by $q_s^*(x,y)$ and for $0<x<K_{\eps}(0)-1$,
\begin{equation}\label{sinx}
\frac{x}{K_{\eps}(0)}\lesssim K_\eps(0)\sin \Big(\frac{\pi x}{K_{\eps}(0)}\Big).
\end{equation}
It follows from (\ref{lem15I1})--(\ref{sinx}) that for $0<x<K_{\eps}(0)-1$ and $\eps$ sufficiently small,
\begin{align}\label{I1}
I_1
&\lesssim xK_{\eps}(0)^2e^{\rho(x-2K_{\eps}(0))}\int_{0}^{K_\eps(0)}e^{\rho y}\sin^2\Big(\frac{\pi y}{K_{\eps}(0)}\Big)\frac{K_\eps(0)-y}{K_\eps(0)}dy\nonumber\\
&\asymp \frac{x}{K_{\eps}(0)}e^{\rho (x-K_{\eps}(0))}\nonumber\\
&\lesssim K_\eps(0)\sin\left(\frac{\pi x}{K_\eps(0)}\right)e^{\rho(x-K_\eps(0))}.
\end{align}
Similarly for $I_3$, by equations (\ref{L0s}), (\ref{green}), (\ref{sinx}) and Fubini's theorem, we get
\begin{align}\label{I3}
I_3
&\lesssim \frac{1}{K_{\eps}(0)^2}e^{-2\rho K_{\eps}(0)}\int_{0}^{K_{\eps}(0)}y^2e^{2\rho y}\int_{0}^{\oldC{C_b}K_{\eps}(0)^2}q_s^*(x,y)dsdy\nonumber\\
&\lesssim  \frac{x}{K_{\eps}(0)^3}e^{\rho (x-2K_{\eps}(0))}\int_{0}^{K_{\eps}(0)}y^2e^{\rho y}(K_{\eps}(0)-y)dy\nonumber\\
&\asymp \frac{x}{K_{\eps}(0)}e^{\rho (x-K_{\eps}(0))}\nonumber\\
&\lesssim K_\eps(0)\sin\left(\frac{\pi x}{K_\eps(0)}\right)e^{\rho(x-K_\eps(0))}.
\end{align}

For $I_2$, by Lemma \ref{densityv}, we have
\begin{align}\label{I2int}
I_2
&\lesssim \int_{\oldC{C_b}K_\eps(0)^2}^{\beta t}\int_{0}^{K_{\eps}(s)}\frac{1}{K_{\eps}(0)^{1/2}K_{\eps}(s)^{1/2}}e^{\rho (x-y)-\sqrt{2}(K_\eps(0)-K_\eps(s))}\sin\Big(\frac{\pi x}{K_{\eps}(0)}\Big)\sin\Big(\frac{\pi y}{K_{\eps}(s)}\Big)\nonumber\\
&\hspace{0.2in}\times K_{\eps}(s)^2\sin^2\Big(\frac{\pi y}{K_{\eps}(s)}\Big)e^{2\rho (y-K_{\eps}(s))}dy ds\nonumber\\
&\asymp \frac{1}{K_{\eps}(0)^{1/2}}e^{\rho (x-K_{\eps}(0))}\sin \Big(\frac{\pi x}{K_{\eps}(0)}\Big)\int_{\oldC{C_b}K_{\eps}(0)^2}^{\beta t}K_{\eps}(s)^{3/2}e^{-\rho K_{\eps}(s)}\int_{0}^{K_{\eps}(s)}e^{\rho y}\sin^3\Big(\frac{\pi y}{K_{\eps}(s)}\Big)dyds\nonumber\\
&\asymp  \frac{1}{K_{\eps}(0)^{1/2}}e^{\rho (x-K_{\eps}(0))}\sin \Big(\frac{\pi x}{K_{\eps}(0)}\Big)\int_{\oldC{C_b}K_{\eps}(0)^2}^{\beta t}\frac{1}{K_{\eps}(s)^{3/2}}ds.
\end{align}
After the change of variable $z=F^{-1}(\eps^{3/2}(t-s))$, since $t\lesssim\eps^{-3/2}$, we see that
\begin{align}\label{L3/2int}
\int_{\oldC{C_b}K_\eps(0)^2}^{\beta t}\frac{1}{K_{\eps}(s)^{3/2}}ds
&=\eps^{-3/4}\int_{F^{-1}(\eps^{3/2}(1-\beta)t)}^{F^{-1}(\eps^{3/2}(t-\oldC{C_b}K_\eps(0)^2))}\frac{z^{1/2}}{\omega^2+z^{2}}dz\nonumber\\
&\asymp \eps^{-3/4}\int_{F^{-1}(\eps^{3/2}(1-\beta)t)}^{F^{-1}(\eps^{3/2}(t-\oldC{C_b}K_\eps(0)^2))}z^{1/2}dz\nonumber\\
&\asymp K_{\eps}(\oldC{C_b}K_\eps(0)^2)^{3/2}-K_{\eps}\left(\beta t\right)^{3/2}.
\end{align}
By (\ref{I2int}) and (\ref{L3/2int}), we have
\begin{equation}\label{I2}
I_2\lesssim \frac{K_\eps(\oldC{C_b}K_\eps(0)^2)^{3/2}}{K_{\eps}(0)^{1/2}}e^{\rho (x-K_{\eps}(0))}\sin \Big(\frac{\pi x}{K_{\eps}(0)}\Big) \leq K_\eps(0)\sin\left(\frac{\pi x}{K_\eps(0)}\right)e^{\rho(x-K_\eps(0))}.
\end{equation}
Similarly, for $I_4$, by Lemma \ref{densityv} and equation (\ref{L3/2int}), we have
\begin{align}\label{I4}
I_4
&\lesssim \int_{\oldC{C_b}K_{\eps}(0)^2}^{\beta t}\int_{0}^{K_{\eps}(s)}\frac{1}{K_{\eps}(0)^{1/2}K_{\eps}(s)^{1/2}}e^{\rho (x-y)-\sqrt{2}(K_\eps(0)-K_\eps(s))}\sin\Big(\frac{\pi x}{K_\eps(0)}\Big)\sin\Big(\frac{\pi y}{K_{\eps}(s)}\Big)\nonumber\\
&\hspace{0.2in}\times \frac{y^2e^{2\rho(y-K_{\eps}(s))}}{K_{\eps}(s)^2}dyds\nonumber\\
&\asymp \frac{1}{K_{\eps}(0)^{1/2}}e^{\rho (x-K_\eps(0))}\sin \Big(\frac{\pi x}{K_{\eps}(0)}\Big)\int_{\oldC{C_b}K_{\eps}(0)^2}^{\beta t}\frac{1}{K_{\eps}(s)^{5/2}}e^{-\rho K_{\eps}(s)}\int_{0}^{K_{\eps}(s)}y^2e^{\rho y}\sin\Big(\frac{\pi y}{K_{\eps}(s)}\Big)dyds\nonumber\\
&\asymp \frac{1}{K_{\eps}(0)^{1/2}}e^{\rho (x-K_{\eps}(0))}\sin \Big(\frac{\pi x}{K_{\eps}(0)}\Big)\int_{\oldC{C_b}K_{\eps}(0)^2}^{\beta t}\frac{1}{K_{\eps}(s)^{3/2}}ds\nonumber\\
&\lesssim K_\eps(0)\sin\left(\frac{\pi x}{K_\eps(0)}\right)e^{\rho(x-K_\eps(0))}.
\end{align}
Finally, Lemma \ref{lem16} follows from equations (\ref{R2I1234}), (\ref{firstmomterm}), (\ref{I1}), (\ref{I3}), (\ref{I2}) and (\ref{I4}).
\qedwhite
\\

Using the first and second moment estimates in Lemmas \ref{lem15lb} and \ref{lem16}, we can control the probability that particles hit the curve $K_\eps(\cdot)$ during the time interval $[\alpha t, \beta t]$.
\begin{Cor}\label{cor17}
Suppose (\ref{limtz}) holds and $t\geq A_{\alpha,\beta}$. There exist positive constants $\newC\label{C_18}$ and $\newC\label{C_19}$ such that for $x\in (0,K_\eps(0)-1)$ and $\eps$ sufficiently small,
\[
\oldC{C_18}K_\eps(0)\sin\left(\frac{\pi x}{K_\eps(0)}\right)e^{\rho(x-K_\eps(0))}\leq P_x^{K_{\eps}}\left(R_t^{K_\eps}(\alpha t, \beta t)>0\right)\leq \oldC{C_19}K_\eps(0)\sin\left(\frac{\pi x}{K_\eps(0)}\right)e^{\rho(x-K_\eps(0))}.
\]
\end{Cor}

\noindent\textit{Proof.}
We first observe that $\tau^{K_\eps}(\alpha t, \beta t)\lesssim K_\eps(0)$ by Lemma \ref{tau(r,s)}.
According to Lemma \ref{lem15lb}, we have
\begin{equation}\label{cor17ub}
P_x^{K_{\eps}}\left(R_t^{K_\eps}(\alpha t,\beta t)>0\right)\leq E_{x}^{K_{\eps}}\left[R_t^{K_\eps}(\alpha t,\beta t)\right]\lesssim K_\eps(0)\sin\left(\frac{\pi x}{K_\eps(0)}\right)e^{\rho(x-K_\eps(0))}.
\end{equation}
On the other hand, we have $\tau^{K_\eps}(\alpha t, \beta t)\gtrsim K_\eps(0)$ by equations (\ref{F-1asymp3/5}), (\ref{tauasymp2/3}) and (\ref{taull2/3}). By Lemmas \ref{lem15lb} and \ref{lem16} and the Cauchy-Schwarz inequality, we have
\begin{equation}\label{cor17lb}
P_x^{K_{\eps}}\left(R_t^{K_\eps}(\alpha t,\beta t)>0\right)\geq\frac{E_{x}^{K_{\eps}}[R_t^{K_\eps}(\alpha t,\beta t)]^2}{E_{x}^{K_{\eps}}[R^{K_\eps}_t(\alpha t, \beta t)^2]}\gtrsim K_\eps(0)\sin\left(\frac{\pi x}{K_\eps(0)}\right)e^{\rho(x-K_\eps(0))}.
\end{equation}
Corollary \ref{cor17} follows from (\ref{cor17ub}) and (\ref{cor17lb}).
\qedwhite

\subsection{Moment estimates of $R_t^{H_\eps}(r,s)$ and proof of Proposition \ref{prop-2}}
\begin{Lemma}\label{lemRH1st}
Suppose (\ref{tles}) holds. If $0\leq v\leq r<s \leq t$, then for $x\in (0, H_\eps(v))$, we have
\begin{align*}
E_{v,x}^{H_{\eps}}\left[R_{t}^{H_\eps}(r,s)\right]
&\lesssim e^{\rho(x-H_\eps(v))}\frac{H_\eps(v)}{H_\eps(s)}\bigg(\tau^{H_\eps}(r,s)\sin\left(\frac{\pi x}{H_{\eps}(v)}\right)\\
&\hspace{0.2in}+\min\left\{\frac{x}{H_\eps(v)}, J_{\tau^{H_\eps}(v,r)}\sin\left(\frac{\pi x}{H_\eps(v)}\right)\min\left\{1, \tau^{H_\eps}(r,s)\right\}\right\}\bigg).
\end{align*}
\end{Lemma}
\noindent\textit{Proof.} The proof is essentially the same as the proof of Lemma \ref{lem15}. By doing the same calculations as (\ref{manygirhit})-(\ref{EvxR(r,s)int}), applying (\ref{densityBrownianH}) instead of (\ref{densityBrownian}), and Lemma \ref{newtau(r,s)} instead of Lemma \ref{tau(r,s)}, we get
\begin{align}\label{EvxR(r,s)intH}
E_{v,x}^{H_\eps}\left[R_t^{H_\eps}(r,s)\right]
&\asymp e^{\rho(x-H_\eps(v))}\int_r^s\frac{H_\eps(v)}{H_\eps(u)^{3}} e^{\pi^2\tau^{H_\eps}(v,u)/2}\times\bigg(-\frac{1}{2}\frac{\partial}{\partial y}\omega_{\tau^{H_\eps}(v,u)}\left(\frac{x}{H_\eps(v)}, y\right)\bigg|_{y=1}\bigg)du.
\end{align}
Note that $H_\eps(u)\geq H_\eps(s)$ for all $r\leq u\leq s$. After the change of variable $z=\tau^{H_\eps}(v,u)$, we get
\begin{equation}\label{lemRH1stsub}
E_{v,x}^{H_\eps}\left[R_t^{H_\eps}(r,s)\right]
\lesssim e^{\rho(x-H_\eps(v))}\frac{H_\eps(v)}{H_\eps(s)}\int_{\tau^{H_\eps}(v,r)}^{\tau^{H_\eps}(v,s)}e^{\pi^2z/2}\bigg(-\frac{1}{2}\frac{\partial}{\partial y}\omega_{z}\left(\frac{x}{H_\eps(v)}, y\right)\bigg|_{y=1}\bigg)dz.
\end{equation}
The lemma follows from equations (\ref{lem7.1}) and (\ref{lemRH1stsub}).
\qedwhite
\\

Suppose (\ref{limtzinf}) holds. Fix $\oldC{C_shorttime}/2\leq\mu<\lambda\leq\oldC{C_shorttime}-1$. Let
\[
u_1=t-\lambda\eps^{-3/2}\qquad \mbox{and}\qquad u_2=t-\mu\eps^{-3/2}.
\]
In the next two lemmas, we are going to estimate the first and second moments of the number of particles hitting $H_\eps(\cdot)$ during the time interval $(u_1, u_2)$ when the initial particle starts out at the position $x$ close to $H_{\eps}(0)$.  Note that when we write $x = H_{\eps}(0) - O(1)$, this means the position $x$ of the initial particle depends on $\eps$.

\begin{Lemma}\label{lemRH1stlb}
Suppose (\ref{limtzinf}) holds. Then for $x= H_\eps(0)-O(1)$,
\[
E_{x}^{H_{\eps}}\left[R_{t}^{H_\eps}(u_1,u_2)\right]\asymp 1.
\]
\end{Lemma}

\noindent\textit{Proof.}
We first apply Lemma \ref{lemRH1st} with $v=0$, $r=u_1$, $s=u_2$ and $x=H_\eps(0)-O(1)$,  and get
\begin{equation}\label{8013}
E_{x}^{H_\eps}\left[R_t^{H_\eps}(u_1,u_2)\right]\lesssim \frac{H_\eps(0)}{H_\eps(u_2)}\sin\left(\frac{\pi x}{H_\eps(0)}\right)\left(\tau^{H_\eps}(u_1,u_2)+J_{\tau^{H_\eps}(0,u_1)}\right).
\end{equation}
By Lemma \ref{newtau(r,s)} and equation (\ref{LL*1}), we have
\begin{equation}\label{8023}
\tau^{H_\eps}(u_1,u_2)\lesssim H_\eps(u_1) = \eps^{-1/2}F^{-1}(\lambda)+O(1)\asymp \eps^{-1/2}\asymp H_\eps(u_2).
\end{equation}
Also, by equation (\ref{L>L*}), we observe that for $\eps$ sufficiently small,
\begin{equation}\label{tau(0,u1)>}
\tau^{H_\eps}(0,u_1) \geq \int_{t-(\lambda+1) \eps^{-3/2}}^{t-\lambda\eps^{-3/2}}\frac{1}{H_\eps(u)^2}du \geq \int_{\lambda \eps^{-3/2}}^{(\lambda+1)\eps^{-3/2}} \frac{1}{(\eps^{-1/2}F^{-1}(\eps^{3/2}u)+O(1))^2} du\asymp \eps^{-1/2}.
\end{equation}
It follows from (\ref{8013}), (\ref{8023}), and (\ref{tau(0,u1)>}) that
\[
E_{x}^{H_{\eps}}\left[R_{t}^{H_\eps}(u_1,u_2)\right]\lesssim 1.
\]

For the lower bound, we apply equation (\ref{EvxR(r,s)intH}). Note that $H_\eps(u)\leq H_\eps(u_1)$ for $u_1\leq u\leq u_2$. After doing the change of variable $z=\tau^{H_\eps}(v,u)$ and applying (\ref{lem7.1}), we get for $x=H_\eps(0)-O(1)$,
\begin{align}\label{RH1stlb1}
E_{x}^{H_\eps}\left[R_t^{H_\eps}(u_1,u_2)\right]
&\gtrsim e^{\rho(x-H_\eps(0))}\frac{H_\eps(0)}{H_\eps(u_1)}\int_{\tau^{H_\eps}(0,u_1)}^{\tau^{H_\eps}(0,u_2)} e^{\pi^2 z/2}\left(-\frac{1}{2}\frac{\partial}{\partial y}\omega_z \left(\frac{x}{K_\eps(0)}, y\right)\bigg|_{y=1}\right) dz \nonumber\\
&\gtrsim \frac{H_\eps(0)}{H_\eps(u_1)}\bigg(\pi\tau^{H_\eps}(u_1,u_2)\sin\bigg(\frac{\pi x}{H_\eps(0)}\bigg)-\oldC{C_mai}J_{\tau^{H_\eps}(0,u_1)}\sin \left(\frac{\pi x}{H_\eps(0)}\right)\bigg).
\end{align}
By doing the same calculations as (\ref{tau(0,u1)>}), we have
\begin{equation}\label{tau(u1,u2)>}
\tau^{H_\eps}(u_1, u_2) \gtrsim \eps^{-1/2}\asymp H_\eps(u_1).
\end{equation}
By equations (\ref{tau(0,u1)>})-(\ref{tau(u1,u2)>}), we have for $x=H_\eps(0)-O(1)$
\[
E_{x}^{H_\eps}\left[R_t^{H_\eps}(u_1,u_2)\right]\gtrsim \frac{H_\eps(0)}{H_\eps(u_1)}\tau^{H_\eps}(u_1,u_2)\sin\bigg(\frac{\pi x}{H_\eps(0)}\bigg)\asymp 1.
\]
\qedwhite

\begin{Lemma}\label{lemRH2nd}
Suppose (\ref{limtzinf}) holds. Then for $x= H_\eps(0)-O(1)$,
\[
E_{x}^{H_{\eps}}\left[R_{t}^{H_\eps}(u_1,u_2)^2\right]\lesssim 1.
\]
\end{Lemma}
\noindent\textit{Proof.} The proof is similar to the proof of Lemma 5.10 in \cite{ms20}.  We use standard second moment estimates as in \eqref{R2I1234}.
Fix a positive constant $\newC\label{C_2ndm}$. Letting $u_3=u_1-\oldC{C_2ndm}/\eps$, we get
\begin{align}\label{RHI012}
E_{x}^{H_{\eps}}\left[R_t^{H_\eps}(u_1, u_2)^{2}\right]
&\leq E_{x}^{H_{\eps}}\left[R_t^{H_\eps}(u_1, u_2)\right]+2\int_{0}^{u_3}\int_{0}^{H_{\eps}(s)}q_{0,s}^{H_\eps}(x,y)\left(E_{s,y}^{H_{\eps}}\left[R_t^{H_\eps}(u_1, u_2)\right]\right)^2dyds\nonumber\\
&\hspace{0.2in}+2\int_{u_3}^{u_2}\int_{0}^{H_{\eps}(s)}q_{0,s}^{H_\eps}(x,y)\left(E_{s,y}^{H_{\eps}}\left[R_t^{H_\eps}(s, u_2)\right]\right)^2dyds.
\end{align}
Next, we obtain upper and lower bounds for $(E_{s,y}^{H_\eps}[R_t^{H_\eps}(u_1,u_2)])^2$ when $0\leq s\leq u_3$, and $(E_{s,y}^{H_\eps}[R_t^{H_\eps}(s,u_2)])^2$ when $u_3<s\leq u_2$. When $0\leq s\leq u_3$, by Lemma \ref{lemRH1st}, we have
\[
\left(E_{s,y}^{H_\eps}\left[R_t^{H_\eps}(u_1,u_2)\right]\right)^2
\lesssim e^{2\rho(y-H_\eps(s))}\frac{H_\eps(s)^2}{H_\eps(u_2)^2}\sin^2\left(\frac{\pi y}{H_{\eps}(s)}\right)\tau^{H_\eps}(u_1,u_2)^2\left(1+J_{\tau^{H_\eps}(s,u_1)}\right)^2.
\]
We note that for $0\leq s\leq u_3$ and $\eps$ sufficiently small,
\[
\tau^{H_\eps}(s,u_1)\geq \int_{u_3}^{u_1}\frac{1}{H_\eps(v)^2}dv\geq \frac{\oldC{C_2ndm}}{F^{-1}_\eps(\lambda+\eps^{1/2}\oldC{C_2ndm})^2}\asymp 1,
\]
which implies that $J_{\tau^{H_\eps}(s,u_1)}\lesssim 1$.  Therefore, using (\ref{8023}), we have for all $0\leq s\leq u_3$,
\begin{align}\label{RH2ndE1}
\left(E_{s,y}^{H_\eps}\left[R_t^{H_\eps}(u_1,u_2)\right]\right)^2
&\lesssim e^{2\rho(y-H_\eps(s))}\frac{\tau^{H_\eps}(u_1,u_2)^2}{H_\eps(u_2)^2}H_\eps(s)^2\sin^2\left(\frac{\pi (H_\eps(s)-y)}{H_\eps(s)}\right)\nonumber\\
&\lesssim e^{2\rho(y-H_\eps(s))}(H_\eps(s)-y)^2.
\end{align}
When $u_3<s\leq u_2$, by Lemma \ref{lemRH1st}, we have
\[
\left(E_{s,y}^{H_\eps}\left[R_t^{H_\eps}(s,u_2)\right]\right)^2 \lesssim e^{2\rho(y-H_\eps(s))}H_\eps(s)^2\sin^2\left(\frac{\pi y}{H_\eps(s)}\right)\left(\frac{\tau^{H_\eps}(s,u_2)}{H_\eps(u_2)}\right)^2+e^{2\rho(y-H_\eps(s))}\frac{y^2}{H_\eps(u_2)^2}.
\]
By equation (\ref{L>L*}) and Lemma \ref{tau(r,s)}, we see that for $u_3\leq s\leq u_2$,
\[
\tau^{H_\eps}(s,u_2)\leq \tau^{K_\eps}(u_3,u_2)\lesssim H_\eps(u_3)\asymp \eps^{-1/2}\asymp H_\eps(u_2).
\]
Also, for $u_3<s\leq u_2$ and $0\leq y\leq H_\eps(s)$, we have
\[
\frac{y}{H_\eps(u_2)}\leq \frac{H_\eps(u_3)}{H_\eps(u_2)}\asymp 1.
\]
Combining the above three equations, we get
\begin{align}\label{RH2ndE2}
\left(E_{s,y}^{H_\eps}\left[R_t^{H_\eps}(s,u_2)\right]\right)^2
&\lesssim e^{2\rho(y-H_\eps(s))}\left(H_\eps(s)^2\sin^2\left(\frac{\pi (H_\eps(s)-y)}{H_\eps(s)}\right)+1\right)\nonumber\\
&\lesssim e^{2\rho(y-H_\eps(s))}\left((H_\eps(s)-y)^2+1\right).
\end{align}
By equations (\ref{RHI012})-(\ref{RH2ndE2}), we obtain
\begin{align*}
E_{x}^{H_{\eps}}\left[R_t^{H_\eps}(u_1, u_2)^{2}\right]
&\lesssim E_{x}^{H_{\eps}}\left[R_t^{H_\eps}(u_1, u_2)\right]\\
&\hspace{0.2in}+\int_0^{u_2}\int_0^{H_\eps(s)}q_{0,s}^{H_\eps}(x,y)e^{2\rho(y-H_\eps(s))}\left((H_\eps(s)-y)^2+1\right)dyds.
\end{align*}
By Lemma \ref{lemRH1stlb} and equation (\ref{densityulbsH}), we have
\begin{align*}
E_{x}^{H_{\eps}}\left[R_t^{H_\eps}(u_1, u_2)^{2}\right]
&\lesssim \int_0^{u_2}\int_{0}^{H_\eps(s)}\frac{1}{H_\eps(0)^{1/2}H_\eps(s)^{1/2}}e^{\rho(x-y)-\sqrt{2}\eps s}\omega_{\tau^{H_\eps}(0,s)}\left(\frac{x}{H_\eps(0)}, \frac{y}{H_\eps(s)}\right)\\
&\hspace{0.2in}\times e^{2\rho(y-H_\eps(s))}\left((H_\eps(s)-y)^2+1\right)dyds+O(1)\\
&=e^{\rho(x-H_\eps(0))}\int_0^{u_2}\int_0^{H_\eps(s)}\frac{1}{H_\eps(0)^{1/2}H_\eps(s)^{1/2}}e^{(\sqrt{2}+\eps)(H_\eps(0)-H_\eps(s))-\sqrt{2}\eps s}\\
&\hspace{0.2in}\times e^{\rho(y-H_\eps(s))}\omega_{\tau^{H_\eps}(0,s)}\left(\frac{x}{H_\eps(0)}, \frac{y}{H_\eps(s)}\right)\left((H_\eps(s)-y)^2+1\right)dyds+O(1).
\end{align*}
Note that $\eps(H_\eps(0)-H_\eps(s))\leq \eps H_\eps(0)\lesssim 1$ by (\ref{L_eps<2epst}), and $\omega_u(1-x,1-y)=\omega_u(x,y)$.  Using Lemma~\ref{newtau(r,s)} and equation (\ref{tau(u1,u2)>}), after the change of variable $z=H_\eps(s)-y$, we have for $x=H_\eps(0)-O(1)$,
\begin{align}\label{I1int}
&E_{x}^{H_{\eps}}\left[R_t^{H_\eps}(u_1, u_2)^{2}\right]\nonumber\\
&\lesssim \int_0^{u_2}\int_0^{H_\eps(s)}\frac{H_{\eps}(0)}{H_\eps(s)^2}e^{\pi^2 \tau^{H_\eps}(0,s)/2}\omega_{\tau^{H_\eps}(0,s)}\left(1-\frac{x}{H_\eps(0)}, \frac{z}{H_\eps(s)}\right)(z^2+1)e^{-\rho z}dzds+O(1).
\end{align}

Denote the double integral as $I$.  If $t \asymp \eps^{-3/2}$, then $H_{\eps}(t/2) \asymp H_{\eps}(0)$.  If $t \geq 2u^* \eps^{-3/2}$, then by (\ref{L_eps<2epst}), we have $H_{\eps}(0)/H_{\eps}(t/2) \leq 4$.  Therefore, we can choose a positive constant $\oldC{C_2ndH}$ such that
\begin{equation}\label{C_2ndHcond}
\oldC{C_2ndH} < \frac{1}{8 \oldC{C_mainub}}\qquad \mbox{and} \qquad \frac{\oldC{C_2ndH} H_{\eps}(0)^2}{H_{\eps}(t/2)^2} \leq 1.
\end{equation}  When $t\asymp\eps^{-3/2}$, we see that $\oldC{C_2ndH}H_\eps(0)^2\ll t/2$. When $t \geq u^* \eps^{-3/2}$, by (\ref{tles}) and (\ref{L_eps<2epst}), we have $\oldC{C_2ndH}H_\eps(0)^2\leq \oldC{C_2ndH}(2\eps t)^2<4\oldC{C_mainub}\oldC{C_2ndH}t< t/2.$ Therefore, for all $t$ satisfying (\ref{limtzinf}), we have
\begin{equation}\label{u4}
u_4:=\oldC{C_2ndH}H_\eps(0)^2<\frac{t}{2}.
\end{equation}
Write $I=I_1+I_2$, where $I_1$ is the portion of the double integral for which $0\leq s\leq u_4$ and $I_2$ is the portion of the double integral for which $u_4<s\leq u_2$. In the following, we estimate $I_1$ and $I_2$ separately.

For $I_1$, we start with the change of variable $u=\tau^{H_\eps}(0,s)$. Letting $h(u)$ be the value such that $\tau^{H_\eps}(0,h(u))=u$, we can divide $I_1$ into two pieces by writing
\begin{align}\label{I1J12}
I_1
&=H_\eps(0)\int_0^{\tau^{H_\eps}(0,u_4)}e^{\pi^2 u/2}\int_0^{H_\eps(u_4)/2}(z^2+1)e^{-\rho z}\omega_u\left(1-\frac{x}{H_\eps(0)},\frac{z}{H_\eps(h(u))}\right)dzdu\nonumber\\
&\hspace{0.2in}+H_\eps(0)\int_0^{\tau^{H_\eps}(0,u_4)}e^{\pi^2 u/2}\int_{H_\eps(u_4)/2}^{H_\eps(h(u))}(z^2+1)e^{-\rho z}\omega_u\left(1-\frac{x}{H_\eps(0)},\frac{z}{H_\eps(h(u))}\right)dzdu\nonumber\\
&=:I_{11}+I_{12}.
\end{align}
We use equation (\ref{lem5.1small}) to upper bound $I_{11}$.
Note that $h(\tau^{H_{\eps}}(0, u_4)) = u_4$, so if $0 \leq u \leq \tau^{H_{\eps}}(0, u_4)$, then $h(u) \leq u_4$.  Therefore,
when $0\leq z\leq H_\eps(u_4)/2$ and $0 \leq u \leq \tau^{H_{\eps}}(0, u_4)$, we have
\begin{equation}\label{sup1/2}
\frac{z}{H_\eps(h(u))}\leq \frac{z}{H_\eps(u_4)}\leq \frac{1}{2}.
\end{equation}
Also, by (\ref{C_2ndHcond})
\begin{equation}\label{tauHu4}
 \oldC{C_2ndH}\leq \tau^{H_\eps}(0,u_4)=\int_0^{u_4}\frac{1}{H_\eps(v)^2}dv\leq \frac{\oldC{C_2ndH}H_\eps(0)^2}{H_\eps(t/2)^2}\leq 1.
\end{equation}
Therefore, by (\ref{lem5.1small}), (\ref{sup1/2}), and (\ref{tauHu4}) along with Tonelli's theorem, we have for $x=H_\eps(0)-O(1)$,
\begin{align}\label{J1}
I_{11}
&\leq H_\eps(0)\int_0^{H_\eps(u_4)/2}(z^2+1)e^{-\rho z}\int_0^{1}e^{\pi^2u/2}\sup_{y\in [0,z/H_\eps(u_4)]}\omega_u\left(1-\frac{x}{H_\eps(0)},y\right)dudz\nonumber\\
&\lesssim H_\eps(0)\int_0^{H_\eps(u_4)/2}(z^2+1)e^{-\rho z}\frac{z}{H_\eps(u_4)}\cdot\frac{x}{H_\eps(0)}dz\nonumber\\
&\leq \frac{H_\eps(0)}{H_\eps(u_4)}\int_0^{\infty}(z^2+1)ze^{-\rho z}dz\nonumber\\
&\lesssim 1.
\end{align}
We use equation (\ref{lem5.2}) to upper bound $I_{12}$. After the change of variable $y=z/H_\eps(h(u))$, applying~(\ref{lem5.2}), we get
\begin{align*}
I_{12}
&\lesssim H_\eps(0)^4e^{-H_\eps(u_4)/2}\int_{0}^{\tau^{H_\eps}(0,u_4)}e^{\pi^2 u/2}\int_0^1\omega_u\left(1-\frac{x}{H_\eps(0)}, y\right)dydu\\
&\lesssim H_\eps(0)^4 e^{-H_\eps(u_4)/2}\left(\tau^{H_\eps}(0,u_4)\sin\left(\frac{\pi x}{H_\eps(0)}\right)+\frac{x}{H_\eps(0)}\right).
\end{align*}
Since $H_\eps(u_4)\asymp H_\eps(0)\gtrsim \eps^{-1/2}$, we have
\begin{equation}\label{J2}
I_{12}\ll 1.
\end{equation}
It follows from (\ref{I1J12}), (\ref{J1}) and (\ref{J2}) that
\begin{equation}\label{RHI1}
I_1\lesssim 1.
\end{equation}

Similarly, for $I_2$, by doing the same calculations in (\ref{I1J12})-(\ref{J2}), we have
\begin{align*}
I_2
&=H_\eps(0)\int_{\tau^{H_\eps}(0,u_4)}^{\tau^{H_\eps}(0,u_2)}e^{\pi^2 u/2}\int_0^{H_\eps(u_2)/2}(z^2+1)e^{-\rho z}\omega_u\left(1-\frac{x}{H_\eps(0)},\frac{z}{H_\eps(h(u))}\right)dzdu\\
&\hspace{0.2in}+H_\eps(0)\int_{\tau^{H_\eps}(0,u_4)}^{\tau^{H_\eps}(0,u_2)}e^{\pi^2 u/2}\int_{H_\eps(u_2)/2}^{H_\eps(h(u))}(z^2+1)e^{-\rho z}\omega_u\left(1-\frac{x}{H_\eps(0)},\frac{z}{H_\eps(h(u))}\right)dzdu\\
&\leq H_\eps(0)\int_0^{H_\eps(u_2)/2}(z^2+1)e^{-\rho z}\int_{\oldC{C_2ndH}}^{\tau^{H_\eps}(0,u_2)}e^{\pi^2u/2}\sup_{y\in [0,z/H_\eps(u_2)]}\omega_u\left(1-\frac{x}{H_\eps(0)},y\right)dudz\\
&\hspace{0.2in}+H_\eps(0)^4 e^{-H_\eps(u_2)/2}\int_{\tau^{H_\eps}(0,u_4)}^{\tau^{H_\eps}(0,u_2)}e^{\pi^2 u/2}\int_0^1\omega_u\left(1-\frac{x}{H_\eps(0)}, y\right)dydu.
\end{align*}
It follows from (\ref{lem5.1s}) and (\ref{lem5.2}) that
\begin{align*}
I_2
&\lesssim H_\eps(0)\int_0^{H_\eps(u_2)/2}(z^2+1)e^{-\rho z}\frac{z}{H_\eps(u_2)}\tau^{H_\eps}(0,u_2)\sin\left( \frac{\pi x}{H_\eps(0)}\right)\\
&\hspace{0.2in}+H_\eps(0)^4e^{-H_\eps(u_2)/2}\left(\tau^{H_\eps}(0,u_2)\sin\left(\frac{\pi x}{H_\eps(0)}\right)+\frac{x}{H_\eps(0)}\right).
\end{align*}
We observe that after the change of variable $v=F^{-1}(u)$,
\begin{align}\label{maxtau}
\tau^{H_\eps}(0,u_2)\leq \tau^{K_\eps}(0,u_2)\leq \eps^{-1/2}\int_0^\infty\frac{1}{F^{-1}(u)^2}du=\eps^{-1/2}\int_0^{\infty}\frac{1}{v^2+\omega^2}dv = \frac{\pi}{2 \omega} \eps^{-1/2}.
\end{align}
Using also the fact that $H_\eps(u_2)\asymp \eps^{-1/2}$, we have for $x=H_\eps(0)-O(1)$
\begin{equation}\label{RHI2}
I_2\lesssim H_\eps(0)\sin\left(\frac{\pi x}{H_\eps(0)}\right)\int_0^\infty (z^2+1)ze^{-\rho z}dz+o(1)\lesssim 1.
\end{equation}

Finally, the lemma follows from (\ref{I1int}), (\ref{RHI1}) and (\ref{RHI2}).
\qedwhite
\\

\begin{proof}[Proof of Proposition \ref{prop-2}]
Suppose (\ref{limtzinf}) holds. By Lemmas \ref{lemRH1stlb} and \ref{lemRH2nd}, and the Cauchy-Schwarz inequality,  we have for $x=H_\eps(0)-1$ and $\eps$ sufficiently small,
\[
P_{x}^{H_{\eps}}\left(R_t^{H_\eps}(u_1, u_2)>0\right) \geq \frac{(E_x^{H_\eps}[R_t^{H_\eps}(u_1,u_2)])^2}{E_x^{H_\eps}[R_t^{H_\eps}(u_1,u_2)^2]}\gtrsim 1
\]
and the proposition follows.
\end{proof}

\section{Proofs of main results}

\subsection{Proof of Proposition \ref{propLzt}}\label{pfpropLsurviv}
In this section, we prove Proposition \ref{propLzt}. We first show that if a particle starts from $L_\eps^*(t)$ at time $0$, the probability that at least one descendant of this particle will survive until time $t$ is bounded from below. Next, we can find a large constant $\oldC{C_th2}$ such that if the process starts from a single particle at $L_\eps^*(t)+\oldC{C_th2}$, there are a large number of offspring with positions above $L_\eps^*(t)$ after $O(1)$ time. Each of them will have at least one descendant alive until time $t$ with some nonzero probability. Taking the number of particles into consideration, the process will survive until time $t$ with very high probability and equation (\ref{Lztub}) follows. A crucial step in the proof of equation (\ref{Lztub}) is the following lemma, whose statement and proof are very similar to Proposition 20 in \cite{bbs14}.

\begin{Lemma}\label{LsurvivSub}
Suppose (\ref{limtz}) holds. Consider the process which starts from a single particle at $x=K_\eps(0)=L_{\eps}^*(t)$. There exist positive constants $\newA\label{A_4}$ and $\newC\label{C_25}$ that only depend on $\oldC{C_shorttime}$ such that for sufficiently small $\eps$, if $t\geq \oldA{A_4}$, then the probability that there are particles alive at time $t$ is bounded below by $\oldC{C_25}$.
\end{Lemma}

\begin{proof}
We claim that it is sufficient to show that there exist positive constants $\oldA{A_4}$ and $\newC\label{C_f}$ that only depend on $\oldC{C_shorttime}$ such that if $t\geq \oldA{A_4}$, then for $\eps$ sufficiently small, the probability that there are particles hitting $K_{\eps}(s)$ for some $s\in [t-\oldA{A_4}, t-\oldA{A_4}/3]$ is bounded below by $\oldC{C_f}$.  To see this, note that a particle that hits $K_{\eps}(s)$ for some $s\in [t-\oldA{A_4}, t-\oldA{A_4}/3]$ has a probability that is bounded away from zero of surviving until time $t$.

We prove this claim by relating the probability that there are particles hitting $K_{\eps}(s)$ for $s\in [t-\oldA{A_4}, t-\oldA{A_4}/3]$ in the slightly subcritical BBM with absorption to the survival probability of a time-inhomogeneous Galton-Watson process. To begin with, we set up constants that will be used in the construction of the Galton-Watson process. Set $\alpha=1/3$ and $\beta=2/3$. Let $A_{1/3,2/3}$ be the constant defined in (\ref{Aalphabeta}) with $\alpha=1/3$ and $\beta=2/3$, and let $\oldC{C_18}$ be the constant defined in Corollary ~\ref{cor17}. For every $y> 0$, consider BBM started with a single particle at the origin where particles move as Brownian motion with drift $-\sqrt{2}$ and are absorbed at the level $-y$. We denote by $N(y)$ the number of particles that are killed upon hitting $-y$. We further denote by $N_{\xi}(y)$ the number of particles that are absorbed at the level $-y$ before time $\xi$. According to equation (5.4) in \cite{neveu}, there exists an almost surely positive random variable $W$ with an infinite expectation such that almost surely
\begin{equation}\label{neveu}
\lim_{y\rightarrow\infty} ye^{-\sqrt{2}y}N(y)=W.
\end{equation}
According to equation (\ref{neveu}), we can choose positive constants $y$, $\newx\label{xi}$ and $M$ large enough that
\begin{equation}\label{consty}
y>\frac{2^{5/3}\omega^2}{c^2}
\end{equation}
and
\begin{equation}\label{AKM}
E\left[ye^{-\sqrt{2} y}(N_{\oldx{xi}}(y)\wedge M)\right]\geq \frac{4}{\oldC{C_18}}.
\end{equation}
Finally, we choose a positive constant $\oldA{A_4}$ large enough that
\begin{equation}\label{Axiy}
\oldA{A_4}> \max\left\{\frac{y^3}{c^3}, \oldx{xi}+A_{1/3,2/3}, \oldx{xi}^{3/2}+\oldx{xi}, \oldx{xi}+\left(\frac{2y}{c}\right)^3\right\}.
\end{equation}

Now we construct the Galton-Watson process. We use a sequence of sets $(T_n)_{n\in \mathbb{N}}$ to record the times at which particles hit the right boundary $K_{\eps}(\cdot)$ in each generation of the branching process. We define $(T_n)_{n\in \mathbb{N}}$  inductively. Let $T_0=\{0\}$.  In the $n$-th generation of the process, we have $T_n=\{t_{n,1}, t_{n,2},..., t_{n,m_n}\}$, which means there are particles hitting $K_{\eps}(t_{n,i})$ at time $t_{n,i}$ for $i=1,...,m_n$. For $i=1,...,m_n$, if $t_{n,i}\geq t-\oldA{A_4}$, then $t_{n,i}\in T_{n+1}$. Otherwise, we keep track of the descendants of the particle that hits $K_{\eps}(t_{n,i})$ at time $t_{n,i}$ until either time $t_{n,i}+\oldx{xi}$, or until particles hit $K_{\eps}(t_{n,i})-y-\eps (r-t_{n,i})$ for some $t_{n,i}<r<t_{n,i}+\oldx{xi}$.  Note that for all $t_{n,i}<r<t_{n,i}+\oldx{xi}$, by (\ref{F-1asymp3/5}) and (\ref{Axiy}), we have for $\eps$ sufficiently small,
\[
K_{\eps}(t_{n,i})-y-\eps (r-t_{n,i})\geq K_{\eps}\left(t-\oldA{A_4}\right)-y-\eps \oldx{xi}\geq c\oldA{A_4}^{1/3}-y-\eps \oldx{xi}> 0.
\]
We denote by $l_{n,i}$ the number of times at which descendants of the $i$-th particle hit the curve $K_{\eps}(t_{n,i})-y-\eps(\cdot-t_{n,i})$ before time $t_{n,i}+\oldx{xi}$, and we denote by $(r_{n,i,j})_{j=1,..., l_{n,i}}$ the corresponding sequence of times. For $j=1,...,l_{n,i}\wedge M$, if one of the descendants of the particles that hits the curve at time $r_{n,i,j}$ proceeds to hit the curve $K_{\eps}(\cdot)$ at some time $v\in [r_{n,i,j}+(t-r_{n,i,j})/3,r_{n,i,j}+2(t-r_{n,i,j})/3]$ afterwards, we put the smallest $v$ at which this happens into the set $T_{n+1}$. For every $n$, define $Z_n=|T_n|$ to be the number of elements in the set $T_n$.

Next, we are going to control the first and second moments of $Z_1$. When $n=0,$ we see that $m_n=1$ and $t_{0,1}=0$. For simplicity, we omit the first and second indices in the subscript.  That is, we write $r_j=r_{0,1,j}$ and $l=l_{0,1}$. For $j=1,..., l\wedge M$, let $I_j$ be the indicator of the event that the $j$-th particle that reaches $K_{\eps}(0)-y-\eps r_j$ at time $r_j<\oldx{xi}$ has at least one descendant that hits the curve $K_{\eps}(\cdot)$ between times $r_j+(t-r_j)/3$ and $r_j+2(t-r_j)/3$. Then we have
\begin{equation}\label{Zxi}
Z_1=\sum_{j=1}^{l\wedge M} I_j.
\end{equation}
We can compute the expectation of $I_j$ using estimates in Section \ref{numberhitting}. Define $\mathcal{G}$ to be the $\sigma$-field generated by  $(r_{j})_{j=1,..., l\wedge M}$.  Recalling that $x = K_{\eps}(0)$, we have
\[
E_{x}(I_j|\mathcal{G}) = P_{x-y-\eps r_j}^{K_{\eps}}\left(R_{t-r_j}\left(\frac{t-r_j}{3},\frac{2(t-r_j)}{3}\right)>0\right)
\]
and $t-r_j\geq \oldA{A_4}-\oldx{xi}>A_{1/3,2/3}$ by (\ref{Axiy}).
It follows from Corollary \ref{cor17} that for $x=K_\eps(0)$, $t\geq \oldA{A_4}$ and $\eps$ sufficiently small,
\begin{align}\label{Ih}
\oldC{C_18} K_{\eps}(r_j)&\sin\left(\frac{\pi (x-y-\eps r_j)}{K_{\eps}(r_j)}\right)e^{\rho(x-\eps r_j-K_{\eps}(r_j))-\eps y-\sqrt{2}y}\nonumber\\
&\leq E_{x}(I_j|\mathcal{G})\nonumber\\
&\hspace{0.5in}\leq \oldC{C_19}K_{\eps}(r_j)\sin\left(\frac{\pi (x-y-\eps r_j)}{K_{\eps}(r_j)}\right)e^{\rho(x-\eps r_j-K_{\eps}(r_j))-\eps y-\sqrt{2}y}.
\end{align}
To evaluate the upper and lower bounds in (\ref{Ih}), we first observe that the function $g(r)=x-\eps r-K_{\eps}(r)$ is an increasing function of $r$. Therefore, if $\eps$ is sufficiently small, then for all $r_j\leq \oldx{xi}$,
\begin{equation}\label{hexplb}
e^{\rho (x-\eps r_j-K_{\eps}(r_j))-\eps y}\geq e^{\rho g(0) - \eps y}\geq \frac{1}{2}.
\end{equation}
For the upper bound, by (\ref{L'}), (\ref{mvt}),  (\ref{F-1asymp3/5}) and (\ref{Axiy}), if $\eps$ is sufficiently small, then for all $r_j\leq \oldx{xi}$,
\begin{equation}\label{ubgr}
g(r_j) \leq g(\oldx{xi})\leq (L_\eps^*)'(t-\oldx{xi})\oldx{xi}-\eps \oldx{xi}
\leq \eps\oldx{xi} \frac{\omega^2}{F^{-1}(\eps^{3/2}(\oldA{A_4}-\oldx{xi}))^2}
\leq \frac{\omega^2\oldx{xi}}{c^2(\oldA{A_4}-\oldx{xi})^{2/3}}\leq \frac{\omega^2}{c^2}.
\end{equation}
Thus if $\eps$ is sufficiently small, then for all $r_j\leq \oldx{xi}$,
\begin{equation}\label{hexpub}
e^{\rho (x-\eps r_j-K_{\eps}(r_j))-\eps y}\leq e^{\omega^2/c^2}.
\end{equation}
Next, we obtain upper and lower bounds on the sine function in (\ref{Ih}). Since the function $g(r)$ is increasing, we get
\begin{equation}\label{hsinub}
\sin\left(\frac{\pi (x-y-\eps r_j)}{K_{\eps}(r_j)}\right)=\sin\left(\frac{\pi (y-g(r_j))}{K_{\eps}(r_j)}\right)\leq \frac{\pi (y-g(r_j))}{K_{\eps}(r_j)}\leq \frac{\pi (y-g(0))}{K_{\eps}(r_j)}=\frac{\pi y}{K_{\eps}(r_j)}.
\end{equation}
For the other direction, we first claim that $0\leq \pi(y-g(r_j))/K_\eps(r_j)\leq \pi/2$. Indeed, by (\ref{consty}) and (\ref{ubgr}), we have for all $r_j\leq \oldx{xi}$ and $\eps$ sufficiently small,
\begin{equation}\label{sin>0}
\frac{\pi (y-g(r_j))}{K_{\eps}(r_j)}\geq  \frac{\pi y}{2K_{\eps}(r_j)}\geq 0.
\end{equation}
Moreover, by (\ref{F-1asymp3/5}) and (\ref{Axiy}), we have for all $r_j\leq \oldx{xi}$ and $\eps$ sufficiently small
\[
\frac{\pi(y-g(r_j))}{K_{\eps}(r_j)}\leq \frac{\pi y}{K_{\eps}(\oldx{xi})}\leq \frac{\pi y}{c(\oldA{A_4}-\oldx{xi})^{1/3}}\leq \frac{\pi}{2}.
\]
From the fact that $\sin(x)\geq 2x/\pi$ for $0\leq x\leq \pi/2$ and equation (\ref{sin>0}), we have for all $r_j\leq \oldx{xi}$ and $\eps$ sufficiently small,
\begin{equation}\label{hsinub2}
\sin\left(\frac{\pi (x-y-\eps r_j)}{K_{\eps}(r_j)}\right)=\sin\left(\frac{\pi (y-g(r_j))}{K_{\eps}(r_j)}\right)\geq \frac{y}{K_{\eps}(r_j)}.
\end{equation}
It follows from equations (\ref{Ih}), (\ref{hexplb}), (\ref{hexpub}), (\ref{hsinub}), and (\ref{hsinub2}) that for all $j=1,..., l$ and $\eps$ sufficiently small
\begin{align}\label{Ihy}
\frac{\oldC{C_18}}{2}ye^{-\sqrt{2}y}\leq &\;E_{x}[I_j|\mathcal{G}]\leq \oldC{C_19}\pi e^{\omega^2/c^2}ye^{-\sqrt{2}y}.
\end{align}
Observe that $l$ has the same distribution as $N_{\oldx{xi}}(y)$. As a result, by (\ref{AKM}), (\ref{Zxi}) and (\ref{Ihy}), we get for $t\geq \oldA{A_4}$ and $\eps$ sufficiently small,
\begin{equation}\label{firstmomZ}
E_{x}[Z_1]\geq \frac{\oldC{C_18}}{2}E\left[ye^{-\sqrt{2}y}(N_{\oldx{xi}}(y)\wedge M)\right]\geq 2.
\end{equation}
For the second moment, by the definition of $Z_1$, we get for all $\eps$,
\begin{equation}\label{secondmomZ}
E_{x}[Z_1^2]\leq M^2.
\end{equation}

For $\oldA{A_4}\leq t\lesssim \eps^{-3/2}$, consider a Galton-Watson process with the offspring distribution
\[
p_t(k)=P(Z_1=k).
\]
Let $q_{t,*}$ be the extinction probability of this process and $q_*=\sup \{q_{t,*}: A_1 \leq  t\leq 2\oldC{C_shorttime} \eps^{-3/2}\}$. According to Lemma 19 in \cite{bbs14}, by (\ref{firstmomZ}) and (\ref{secondmomZ}), we have for all $t\geq \oldA{A_4}$ and $\eps$ sufficiently small,
\[
1-q_{t,*}\geq \frac{2(E_{x}[Z_1]-1)}{E_{x}[Z_1(Z_1-1)]}\geq \frac{2}{M^2},
\]
which gives that
\begin{equation}\label{1-qstar}
1-q_{*}\geq \frac{2}{M^2}.
\end{equation}
Finally, we connect the supremum extinction probability $q_{*}$ with the probability that there are particles hitting $K_{\eps}(s)$ for some $s\in [t-\oldA{A_4}, t-\oldA{A_4}/3]$. Define
\[
q_{t}=\lim_{n\rightarrow\infty}P(T_n=\emptyset).
\]
Note that $1-q_t$ is indeed the probability that there are particles reaching $K_{\eps}(s)$ for some $s\in [t-\oldA{A_4}, t-\oldA{A_4}/3]$. It was shown in the proof of Proposition 20 in \cite{bbs14} that
\begin{equation}\label{qtstar}
q_t\leq q_{*}.
\end{equation}
Equations (\ref{1-qstar}) and (\ref{qtstar}) imply that for $t\geq \oldA{A_4}$ and $\eps$ sufficiently small, the probability that there are particles hitting $K_{\eps}(s)$ for some $s\in [t-\oldA{A_4}, t-\oldA{A_4}/3]$ is bounded below by $2/M^2$ and the lemma follows.
\end{proof}

Now we have all the ingredients to prove Proposition \ref{propLzt}.

\begin{proof}[Proof of Proposition \ref{propLzt}]
If $t< \oldA{A_4}$, then equation (\ref{Lztub}) is obvious. We only consider the case when $t\geq \oldA{A_4}$. Let $\delta>0$. Choose a positive integer $\newC\label{C_8}$ large enough that
\begin{equation}\label{deltaub1}
(1-\oldC{C_25})^{\oldC{C_8}}<\frac{\delta}{2}.
\end{equation}
Consider BBM started with a single particle at the origin where particles move as Brownian motion with drift $-\sqrt{2}$ and are absorbed at the level $-y$. Recall that $N_{\xi}(y)$ denotes the number of particles that are absorbed at $-y$ before time $\xi$. According to (\ref{neveu}), we can choose positive constants $\newx\label{xi_4}$ and $\oldC{C_th2}$ large enough that for $\eps$ sufficiently small
\begin{equation}\label{slightneveu}
P\left(N_{\oldx{xi_4}}(\oldC{C_th2}-\eps \oldx{xi_4})<\oldC{C_8}\right)<\frac{\delta}{2}.
\end{equation}

We denote by $\mathcal{N}_t^{\rho}$ the set of surviving particles at time $t$ and $\{X_u^{\rho}(t), u\in\mathcal{N}_t^{\rho}\}$ the set of positions of particles at time $t$ for BBM with absorption and drift $-\rho$. One can couple the subcritical process with drift $\rho>\sqrt{2}$ with the critical process such that for every $u\in\mathcal{N}_t^{\rho}$, there exists $v\in\mathcal{N}_t^{\sqrt{2}}$ satisfying $X_u^{\rho}(s)=X_v^{\sqrt{2}}(s)-\eps s$ for all $0\leq s\leq t$,  and $$\mathcal{N}_t^{\rho}=\{u\in\mathcal{N}_t^{\sqrt{2}}, X_u^{\sqrt{2}}(s)-\eps s> 0 \mbox{ for all } s\leq t\}.$$ Equation (\ref{slightneveu}) implies that for every $\eps$, if the subcritical BBM with absorption starts with a single particle at $L_\eps^*(t)+\oldC{C_th2}$, then with probability at least $1-\delta/2$, the number of particles hitting $L_{\eps}^*(t) + \eps(\xi_2 - s)$ for some $s \in [0, \xi_2]$ is at least $\oldC{C_8}$. By Lemma \ref{LsurvivSub}, each of them has probability at least $\oldC{C_25}$ of surviving until time $t$. Combining Lemma \ref{LsurvivSub} with equations (\ref{deltaub1}) and (\ref{slightneveu}), if $t\geq \oldA{A_4}$, we have for $\eps$ sufficiently small,
\[
P_{L_\eps^*(t)+\oldC{C_th2}}\left(\zeta<t\right)\leq (1-\oldC{C_25})^{\oldC{C_8}}+P\left(N_{\oldx{xi_4}}(\oldC{C_th2}-\eps \oldx{xi_4})<\oldC{C_8}\right)<\delta.
\]
and equation (\ref{Lztub}) follows.
\end{proof}

\subsection{Proof of Theorem \ref{thmLzt}}

\begin{proof}[Proof of equation (\ref{Lztlb})]
By Lemma \ref{LepsLbar}, it is enough to prove the result with $L_{\eps}(t)$ in place of $\bar{L}_{\eps}(t)$.
Choose a constant $\oldC{C_a}>0$. If $t\lesssim 1$, then $\bar{L}_{\eps}(t) \lesssim 1$, and so equation (\ref{Lztlb}) is obvious. It is therefore enough to consider the case when $t> \oldC{C_a}$. We first consider the case when $\oldC{C_shorttime}\eps^{-3/2}\leq  t \leq \oldC{C_mainub} \eps^{-2}$.  Recall that $\oldC{C_2ndH} < 1/(8\oldC{C_mainub})$. Let
\[
u_4=\oldC{C_2ndH}H_\eps(0)^2, \qquad u_5= t-\frac{\oldC{C_shorttime}}{2}\eps^{-3/2}, \qquad u_6=t-\oldC{C_a},
\]
noting that the definition of $u_4$ was previously given in (\ref{u4}).
Using (\ref{u4}), we have for $\eps$ sufficiently small,
\begin{equation}\label{thmCb}
u_4<\frac{t}{2}\leq u_5<u_6.
\end{equation}
We divide the particles alive at time $t$ into five subsets according to the time at which particles hit the curve $H_\eps(\cdot)$ or $K_\eps(\cdot)$. Recall that we denote by $\mathcal{N}_t^{\rho}$ the set of surviving particles at time $t$ and $\{X_u^{\rho}(t), u\in\mathcal{N}_t^{\rho}\}$ the set of positions of particles at time $t$.
By \eqref{LL*1}, there exists a constant $\newC\label{C_claim}$ such that
\begin{equation}\label{claim1}
H_\eps(u_5)-\oldC{C_claim}\leq K_\eps(u_5)
\end{equation}
Define
\begin{align*}
D_1&=\{\exists u\in\mathcal{N}_t^{\rho}: X_u^{\rho}(s)\geq H_\eps(s) \mbox{ for some } s\in [0,u_4]\},\\
D_2&=\{\exists u\in\mathcal{N}_t^{\rho}\setminus D_1: X_u^{\rho}(s)\geq H_\eps(s) \mbox{ for some } s\in (u_4, u_5]\},\\
D_3&=\{\exists u\in\mathcal{N}_t^{\rho}\setminus (D_1\cup D_2): X_u^{\rho}(u_5)\geq H_\eps(u_5)-\oldC{C_claim}\},\\
D_4&=\{\exists u\in\mathcal{N}_t^{\rho}\setminus(D_1\cup D_2\cup D_3): X_u^{\rho}(s)\geq K_\eps(s)\mbox{ for some } s\in (u_5,u_6]\}, \\
D_5&=\{\mathcal{N}_t^{\rho} \neq \emptyset\}\setminus(D_1\cup D_2\cup D_3\cup D_4).
\end{align*}
We observe that
\begin{align}\label{Lztlbsum}
&P_{L_\eps(t)-\oldC{C_main1}}(\zeta>t)\nonumber\\
&\hspace{0.1in}\leq \sum_{i=1}^5 P_{L_\eps(t)-\oldC{C_main1}}(D_i)\nonumber\\
&\hspace{0.1in}\leq E^{H_\eps}_{L_\eps(t)-\oldC{C_main1}}\left[R_t^{H_\eps}(0,u_4)\right]+E^{H_\eps}_{L_\eps(t)-\oldC{C_main1}}\left[R_t^{H_\eps}(u_4,u_5)\right]+ \int_{H_\eps(u_5)-\oldC{C_claim}}^{H_\eps(u_5)}q^{H_\eps}_{0,u_5}(L_\eps(t)-\oldC{C_main1},y)dy\nonumber\\
&\hspace{0.3in}+\int_0^{H_\eps(u_5)-\oldC{C_claim}}q_{0,u_5}^{H_\eps}(L_\eps(t)-\oldC{C_main1},y)E_{u_5,y}^{K_\eps}\left[R_{t}^{K_\eps}(u_5, u_6)\right]dy
+ P_{L_\eps(t)-\oldC{C_main1}}(D_5) \nonumber\\
&\hspace{0.1in}=: I_1+I_2+I_3+I_4+I_5.
\end{align}

We first estimate $I_1$. By Lemma \ref{lemRH1st} and the fact that $\sin(\pi x)=\sin(\pi(1-x))\leq \pi (1-x)$ for $0\leq x\leq 1$, there exists a positive constant $\newC\label{C_e}$ such that
\[
I_1\leq \oldC{C_e}e^{-\rho \oldC{C_main1}}\frac{H_\eps(0)}{H_\eps(u_4)}\left(\tau^{H_\eps}(0,u_4)\frac{\pi \oldC{C_main1}}{H_\eps(0)}+1\right).
\]
By Lemma \ref{newtau(r,s)}, we see that $\tau^{H_\eps}(0,u_4)\leq H_\eps(0)/\omega^2$ and thus
\[
I_1\leq \oldC{C_e}e^{-\rho \oldC{C_main1}}\left(\frac{\pi\oldC{C_main1}}{\omega^2}+1\right)\frac{H_\eps(0)}{H_\eps(u_4)}.
\]
If $t \geq 2u^*\eps^{-3/2}$, then by (\ref{L_eps<2epst}) and (\ref{thmCb}), we have for $\eps$ sufficiently small,
\[
\frac{H_\eps(0)}{H_\eps(u_4)}\leq \frac{2\eps t}{\eps t/2}=4.
\]
If $t \asymp \eps^{-3/2}$, then by (\ref{L>L*}) and (\ref{LL*1}), we have $H_{\eps}(0) \asymp \eps^{-1/2}$ and thus $u_4 \asymp \eps^{-1}$, which implies that $H_{\eps}(0)/H_{\eps}(u_4) = L_{\eps}(t)/L_{\eps}(t - u_4) \rightarrow 1$ as $\eps \rightarrow 0$.  Therefore, for $\eps$ sufficiently small,
\begin{equation}\label{thmI1}
I_1\leq 4\oldC{C_e}e^{-\rho \oldC{C_main1}}\left(\frac{\pi\oldC{C_main1}}{\omega^2}+1\right).
\end{equation}

For $I_2$, note that by reasoning as in (\ref{taulb}), we get $\tau^{H_{\eps}}(0, u_4) \geq \oldC{C_2ndH}$.  Therefore,
it follows from Lemma \ref{lemRH1st} along with (\ref{L>L*}) and (\ref{maxtau}) that for $\eps$ sufficiently small,
\begin{align}\label{thmI2}
I_2
&\leq \oldC{C_e}e^{-\rho \oldC{C_main1}}\frac{H_\eps(0)}{H_\eps(u_5)}\left(\tau^{H_\eps}(u_4,u_5)+J_{\tau^{H_\eps}(0,u_4)}\right)\sin\left(\frac{\pi \oldC{C_main1}}{H_\eps(0)}\right)\nonumber\\
&\leq \pi\oldC{C_e}\oldC{C_main1}e^{-\rho \oldC{C_main1}}\frac{1}{K_\eps(u_5)}\left(\tau^{H_\eps}(u_4,u_5)+J_{\oldC{C_2ndH}}\right)\nonumber\\
&\leq  \pi\oldC{C_e}\oldC{C_main1}e^{-\rho \oldC{C_main1}}\frac{1}{\eps^{-1/2}F^{-1}(\oldC{C_shorttime}/2)}\left(\frac{\pi}{2\omega}\eps^{-1/2}+J_{\oldC{C_2ndH}}\right)\nonumber\\
&\leq \frac{\pi^2\oldC{C_e}\oldC{C_main1}}{\omega F^{-1}(\oldC{C_shorttime}/2)}e^{-\rho \oldC{C_main1}}.
\end{align}

For $I_3$, we apply Lemma \ref{densityvH}. Since $\eps H_\eps(0)\lesssim 1$ by (\ref{L_eps<2epst}) and $H_\eps(u_5)\asymp \eps^{-1/2}$, we have for $\eps$ sufficiently small,
\begin{align}\label{thmI3}
I_3
&\lesssim\frac{H_\eps(0)}{H_\eps(u_5)^2}e^{\rho(L_\eps(t)-\oldC{C_main1})-\sqrt{2}(H_\eps(0)-H_\eps(u_5))} \sin\left(\frac{\pi \oldC{C_main1}}{H_\eps(0)}\right) \int_{H_\eps(u_5)-\oldC{C_claim}}^{H_\eps(u_5)} e^{-\rho y}\sin\left(\frac{\pi y}{H_\eps(u_5)}\right)dy\nonumber\\
&\leq\frac{H_\eps(0)}{H_\eps(u_5)^2}e^{-\rho\oldC{C_main1}+\eps H_\eps(0)+\sqrt{2}H_\eps(u_5)} \frac{\pi \oldC{C_main1}}{H_\eps(0)}\oldC{C_claim}e^{-\rho (H_\eps(u_5)-\oldC{C_claim})}\sin\left(\frac{\pi \oldC{C_claim}}{H_\eps(u_5)}\right)\nonumber\\
&\ll 1.
\end{align}

For $I_4$, by (\ref{claim1}) and Lemmas \ref{densityvH} and \ref{lem15}, there exists a constant $\newC\label{C_g}$ such that for $\eps$ sufficiently small,
\begin{align*}
I_4
&\leq \oldC{C_g}\int_0^{K_\eps(u_5)}\frac{H_\eps(0)}{H_\eps(u_5)^2}e^{\rho(L_\eps(t)-\oldC{C_main1}-y)-\sqrt{2}(H_\eps(0)-H_\eps(u_5))}\sin\left(\frac{\pi \oldC{C_main1}}{H_\eps(0)}\right)\sin\left(\frac{\pi y}{H_\eps(u_5)}\right)\\
&\hspace{0.2in}\times e^{\rho(y-K_\eps(u_5))}\left(\tau^{K_\eps}(u_5,u_6)\sin\left(\frac{\pi y}{K_\eps(u_5)}\right)+\frac{y}{K_\eps(u_5)}\right)dy\\
&\leq\frac{\pi\oldC{C_main1}\oldC{C_g}}{H_\eps(u_5)^2}e^{-\rho \oldC{C_main1}+\eps H_\eps(0)+\rho \oldC{C_claim}}\Bigg(\tau^{K_\eps}(u_5, u_6)\int_0^{K_\eps(u_5)}\sin\left(\frac{\pi y}{H_{\eps}(u_5)}\right)\sin\left(\frac{\pi y}{K_\eps(u_5)}\right)dy\\
&\hspace{0.2in}+\int_0^{K_\eps(u_5)}\frac{y}{K_\eps(u_5)}\sin\left(\frac{\pi y}{H_\eps(u_5)}\right)dy\Bigg) \\
&\leq \frac{\pi\oldC{C_main1}\oldC{C_g}}{H_\eps(u_5)^2}e^{-\rho \oldC{C_main1}+\eps H_\eps(0)+\rho \oldC{C_claim}}
\big( \tau^{K_{\eps}}(u_5, u_6) K_{\eps}(u_5) + K_{\eps}(u_5) \big).
\end{align*}
Note that $\eps H_{\eps}(0) = \eps L_{\eps}(t) \leq \eps L_{\eps}( C_1 \eps^{-2}) \leq 2C_1$ by (\ref{L_eps<2epst}).  Also, we have $\tau^{K_{\eps}}(u_5, u_6) \leq \frac{\pi}{2 \omega} \eps^{-1/2}$ by (\ref{maxtau}), and $H_{\eps}(u_5) \geq K_{\eps}(u_5) = \eps^{-1/2} F^{-1}(C_4/2)$ by (\ref{L>L*}).  Therefore, for $\eps$ sufficiently small,
\begin{align}\label{thmI4}
I_4 &\leq \frac{\pi\oldC{C_main1}\oldC{C_g}}{K_\eps(u_5)}e^{-\rho \oldC{C_main1}+2\oldC{C_mainub}+\rho \oldC{C_claim}}\left(\frac{\pi }{2\omega}\eps^{-1/2}+1\right)\nonumber\\
&\leq \frac{\pi^2\oldC{C_main1}\oldC{C_g}}{\omega F^{-1}(\oldC{C_shorttime}/2)}e^{-\rho \oldC{C_main1}+2\oldC{C_mainub}+\rho \oldC{C_claim}}.
\end{align}

For $I_5$, by Lemma \ref{densityvH}, there exists a constant $\newC\label{C_h}$ such that for $\eps$ sufficiently small,
\begin{align}\label{thmI5}
I_5
&\leq \int_0^{H_\eps(u_6)}q_{0,u_6}^{H_\eps}(L_\eps(t)-\oldC{C_main1},y)dy \nonumber \\
&\leq \oldC{C_h}\frac{H_\eps(0)}{H_\eps(u_6)^2}e^{\rho(L_\eps(t)-\oldC{C_main1})-\sqrt{2}(H_\eps(0)-H_\eps(u_6))}\sin\left(\frac{\pi \oldC{C_main1}}{H_\eps(0)}\right)\int_0^{H_\eps(u_6)}e^{-\rho y}\sin\left(\frac{\pi y}{H_\eps(u_6)}\right)dy \nonumber \\
&\leq \frac{\pi C_2 C_{21}}{H_{\eps}(u_6)} e^{\eps H_{\eps}(0) + \sqrt{2} H_{\eps}(u_6)} e^{-\rho C_2}.
\end{align}
Because $\eps H_\eps(0)\lesssim 1$ by (\ref{L_eps<2epst}) and $H_{\eps}(u_6) \asymp 1$ by (\ref{L>L*}) and (\ref{LL*1}),
it follows from (\ref{Lztlbsum})-(\ref{thmI5}) that for every $\delta>0$, we can choose $\oldC{C_main1}$ large enough so that (\ref{Lztlb}) holds for $\eps$ sufficiently small.

We next consider the case when $\oldC{C_a}< t\leq\oldC{C_shorttime}\eps^{-3/2}$.  By Markov's inequality, the probability that the process survives until time $t$ is bounded above by the sum of the expected number of particles hitting $K_\eps(\cdot)$ before time $u_6$ and the expected number of particles alive at time $u_6$.  We thus have
\begin{equation}\label{Lztlbsumshort}
P_{L_\eps(t)-\oldC{C_main1}}(\zeta>t)\leq E^{K_\eps}_{L_\eps(t)-\oldC{C_main1}}\left[R_t^{K_\eps}(0,u_6)\right]+\int_0^{H_\eps(u_6)}q^{H_\eps}_{0,u_6}\left(L_\eps(t)-\oldC{C_main1},y\right)dy.
\end{equation}
It follows from (\ref{LL*1}) that there is a positive constant $\newC\label{C_newv8}$ such that when $\oldC{C_a}< t\leq\oldC{C_shorttime}\eps^{-3/2}$, we have $L_{\eps}(t) - L_{\eps}^*(t) \leq \oldC{C_newv8}$.  We may choose $\oldC{C_main1}>2\oldC{C_newv8}$. Then  $L_\eps(t)-\oldC{C_main1}<K_\eps(0)-\oldC{C_main1}/2$.  By Lemmas \ref{tau(r,s)} and \ref{lem15}, for sufficiently small $\eps$,
\begin{align}\label{Lztlbsum1}
E^{K_\eps}_{L_\eps(t)-\oldC{C_main1}}\left[R_t^{K_\eps}(0,u_6)\right] &\leq
e^{\rho(L_{\eps}(t) - C_2 - K_{\eps}(0))} \left( \tau^{K_{\eps}}(0, u_6) \sin \left(\frac{\pi(L_{\eps}(t) - C_2)}{K_{\eps}(0)} \right) + \frac{L_{\eps}(t) - C_2}{K_{\eps}(0)} \right) \nonumber \\
&\leq e^{-\rho C_2/2} \left( \frac{\pi C_2}{\omega^2} + 1 \right).
\end{align}
Because $u_6 \gtrsim H_{\eps}(0)^2$, the second term in (\ref{Lztlbsumshort}) can be bounded as in (\ref{thmI5}).
Therefore, for any $\delta>0$, we can choose $\oldC{C_main1}$ large enough that (\ref{Lztlb}) holds for all $\eps$ sufficiently small.

Finally, we prove equation (\ref{Lztlb}) for all $t$ satisfying (\ref{tles}). Suppose (\ref{Lztlb}) does not hold true. Then there exists $\delta>0$ such that for all $n\in \mathbb{N}$, we can find $0<\eps_n<1$ with
\begin{equation}\label{contradict}
P_{L_{\eps_n}(t_{\eps_n})-n}(\zeta>t_{\eps_n})\geq \delta.
\end{equation}
Since $0\leq \eps^{2}_{n}t_{\eps_n}\leq\oldC{C_mainub}$, we can find a subsequence $\{\eps_{n_k}\}_{k=1}^\infty$ such that $t_{\eps_{n_k}}\leq \oldC{C_a}$, or $\oldC{C_shorttime}\eps_{n_k}^{-3/2}\leq t_{\eps_{n_k}} \leq \oldC{C_mainub}\eps_{n_k}^{-2}$, or $\oldC{C_a}<t_{\eps_{n_k}}\leq \oldC{C_shorttime}\eps_{n_k}^{-3/2}$ for all $k$, and also (\ref{contradict}) holds with $n$ replaced by $n_k$. This contradicts the arguments above and  equation (\ref{Lztlb}) follows.
\end{proof}

\begin{proof}[Proof of (\ref{newLztub})]
By Lemma \ref{LepsLbar}, it is enough to prove the result with $L_{\eps}(t)$ in place of $\bar{L}_{\eps}(t)$.
When $t$ satisfies (\ref{limtz}), equation (\ref{newLztub}) follows from Proposition~\ref{propLzt} and equation (\ref{LL*1}). When $t$ satisfies (\ref{limtzinf}), we choose a positive integer $\newC\label{C_last}$ large enough that
\[
(1-\oldC{C_longhit}\oldC{C_25})^{\oldC{C_last}}<\frac{\delta}{2}.
\]
By using the same argument as the proof of Proposition \ref{propLzt}, we could choose constants $\oldC{C_main2}$ and $\newx\label{xi_5}$ large enough that for $\eps$ sufficiently small,
\[
P\left(N_{\oldx{xi_5}}(\oldC{C_main2}+1-\eps \oldx{xi_5})<\oldC{C_last}\right)<\frac{\delta}{2}.
\]
This implies that if the process starts with a single particle at $L_\eps(t)+\oldC{C_main2}$, then with probability at least $1-\delta/2$, the number of particles hitting $L_\eps(t) - 1 + \eps(\oldx{xi_5} - s)$ for some $s \in [0, \oldx{xi_5}]$ is greater than $\oldC{C_last}$.  By an easy coupling argument, the probability of survival until time $t$ can only decrease if we move these particles to $L_{\eps}(t) - 1$.  By Proposition \ref{prop-2} and Lemma \ref{LsurvivSub}, each of them hits the curve $H_\eps(\cdot)$ during the time $(t-\lambda\eps^{-3/2}, t-\mu\eps^{-3/2})$ with probability at least $\oldC{C_longhit}$, and once it hits the curve $H_\eps(\cdot)$ in the last $O(\eps^{-3/2})$ time, it will survive until time $t$ with probability at least $\oldC{C_25}$. Therefore, we have for $\eps$ sufficiently small,
\[
P_{L_\eps(t)+\oldC{C_main2}}(\zeta > t) \leq (1-\oldC{C_longhit}\oldC{C_25})^{\oldC{C_last}}+P\left(N_{\oldx{xi_5}}(\oldC{C_main2}+1-\eps \oldx{xi_5})<\oldC{C_last}\right)<\delta
\]
and equation (\ref{newLztub}) follows.

Finally, for all $t$ satisfying (\ref{tles}), equation (\ref{newLztub}) can be argued in the same way as equation (\ref{Lztlb}). Suppose (\ref{newLztub}) does not hold true. Then there exists $\delta>0$, such that for all $n\in \mathbb{N}$, we can find $0<\eps_n<1$ with
\begin{equation}\label{contradictub}
P_{L_{\eps_n}(t_{\eps_n})+n}(\zeta<t_{\eps_n})\geq \delta.
\end{equation}
Since $0\leq \eps^{2}_{n}t_{\eps_n}<\oldC{C_mainub}$, we can find a subsequence $\{\eps_{n_k}\}_{k=1}^\infty$ such that $t_{\eps_{n_k}}$ satisfies either (\ref{limtz}) or~(\ref{limtzinf}), and also (\ref{contradictub}) holds with $n$ replaced by $n_k$. This contradicts the proof above and  equation~(\ref{newLztub}) follows.
\end{proof}

\paragraph{Acknowledgements.}
This research was carried out in part during the Workshop \emph{Branching systems, reaction-diffusion equations and population models} organised in May 2022, we wish to thank the organizers for this conference. In particular, BM acknowledge partial support by funding from the Simons Foundation and the Centre de Recherches Mathématiques, through the Simons-CRM scholar-in-residence program.

\end{document}